\documentclass[11pt, reqno]{amsart}
\usepackage{amsfonts}
\usepackage{amsthm,amsmath}
\usepackage{amsbsy}
\usepackage{bm}
\usepackage[margin=2.5cm]{geometry}
\usepackage{amssymb} 
\usepackage{verbatim}
 \usepackage{color}

\usepackage[mathscr]{eucal}
\usepackage{times}

\usepackage{units}
\usepackage{enumerate,enumitem}
\usepackage{bbm}
\usepackage{exscale}
\usepackage{hyperref}
\hypersetup{ colorlinks=true, linkcolor=blue, citecolor=blue,
  filecolor=blue, urlcolor=blue}
  \numberwithin{equation}{section} 
  \date{}

\usepackage[comma, sort&compress]{natbib}

\theoremstyle{plain} 
    \newtheorem{theorem}{Theorem}
    \numberwithin{theorem}{section} 

    \newtheorem{lemma}[theorem]{Lemma}
    \newtheorem{proposition}[theorem]{Proposition}

\theoremstyle{definition} 
    \newtheorem{definition}[theorem]{Definition}

    \newtheorem{remark}[theorem]{Remark}

\DeclareMathOperator{\R}{\mathbb{R}}

\DeclareMathOperator{\Z}{\mathbb{Z}}
\DeclareMathOperator{\N}{\mathbb{N}}

\usepackage{mathtools}
\DeclarePairedDelimiter\ceil{\lceil}{\rceil}
\DeclarePairedDelimiter\floor{\lfloor}{\rfloor}

\DeclareMathOperator{\De}{d}

\newcommand{\e}{\mathrm{e}}

\newcommand{\eps}{\varepsilon}
\newcommand{\vr}{\varphi}
\newcommand{\var}{\mathbf{Var}}

\newcommand{\prob}{\mathbf{P}}
\newcommand{\E}{\mathbf{E}}
\newcommand{\cov}{\mathbf{Cov}}

\newcommand{\la}{\left\langle}
\newcommand{\ra}{\right\rangle}

\begin{document}
\title{Scaling Limit of Semiflexible Polymers: a Phase Transition }

\author[A. Cipriani]{Alessandra Cipriani}
\address{TU Delft (DIAM), Building 28, van Mourik Broekmanweg 6, 2628 XE, Delft, The Netherlands}
\email{A.Cipriani@tudelft.nl}
\author[B. Dan]{Biltu Dan}
\author[R.~S.~Hazra]{Rajat Subhra Hazra}
\thanks{The first author is supported by the grant 613.009.102 of the Netherlands Organisation for Scientific Research (NWO). The third author acknowledges the MATRICS grant from SERB and the Dutch stochastics cluster STAR (Stochastics -- Theoretical and Applied Research) for an invitation to TU Delft where part of this work was carried out. The authors would like to thank Luca Avena and Alberto Chiarini for their remarks that led to the draft of the present article. We are very grateful to an anonymous referee for her/his insightful comments and for outlining a new proof of Theorem~\ref{approx_result}~\ref{thm:gff_one} improving on a previous version.}
\address{ISI Kolkata, 203, B.T. Road, Kolkata, 700108, India}
\email{biltudanmath@gmail.com, rajatmaths@gmail.com}

\date{}

\begin{abstract}
We consider a semiflexible polymer in $\Z^d$ which is a random interface model with a mixed gradient and Laplacian interaction. The strength of the two operators is governed by two parameters called lateral tension and bending rigidity, which might depend on the size of the graph. In this article we show a phase transition in the scaling limit according to the strength of these parameters: we prove that the scaling limit is, respectively, the Gaussian free field, a ``mixed'' random distribution and the continuum membrane model in three different regimes.
\end{abstract}
\maketitle

\section{\bf Introduction}
In this article we study a model which is a special instance of a more general class of random interfaces. Random interfaces are fields $\phi=(\phi_{x})_{x\in\Z^{d}}$, whose distribution is specified by a probability measure on $\mathbb{R}^{\mathbb{Z}^{d}}$,
$d\ge1$. The density is given in terms of an energy function $H$ called Hamiltonian and has the form
\begin{align}\label{eq:measure_def}
\prob_{\Lambda}(\mathrm{d}\phi):=\frac{\mathrm{e}^{-H(\phi)}}{Z_\Lambda}\prod_{x\in{\Lambda}}\mathrm{d}\phi_{x}\prod_{x\in\mathbb{Z}^{d}\setminus \Lambda}
\delta_{0}(\mathrm{d}\phi_{x}),
\end{align}
where $\Lambda\Subset\mathbb{Z}^{d}$ is a finite subset, $\mathrm{d}\phi_{x}$ is the Lebesgue measure on $\mathbb{R}$, $\delta_{0}$ is the Dirac measure at $0,$ and $Z_{\Lambda}$ is a normalizing constant. We are imposing zero boundary conditions: almost surely $\phi_{x}=0$ for all $x\in\mathbb{Z}^{d}\setminus{\Lambda}$, but the definition holds for more general boundary conditions. A special case is when the Hamiltonian is given by
\begin{equation}\label{eq:Ham_def}H(\vr)=\sum_{x\in\Z^d}\left(\kappa_1\|\nabla\vr_x\|^2+\kappa_2(\Delta\vr_x)^2\right)\end{equation}
where $\nabla$ is the discrete gradient and $\Delta$ is the discrete Laplacian defined by 
$$\nabla f(x)=(f(x+e_i)-f(x))_{i=1}^d$$
$$\Delta f(x)= \frac1{2d}\sum_{i=1}^d \left( f(x+e_i) +f(x-e_i) -2f(x) \right)$$
for any $x\in \Z^d$, $f:\Z^d\to\R$, and $\kappa_1,\,\kappa_2$ are two non-negative parameters. In the physics literature, the above Hamiltonian is considered to be the energy of a semiflexible membrane (or semiflexible polymer if $d=1$) where the parameters $\kappa_1$ and $\kappa_2$ are the {\it lateral tension} and the {\it bending rigidity}, respectively (\cite{Leibler:2006, ruiz2005phase, lipowsky1995generic}).
  
When $\kappa_2=0$, the model is the purely gradient model and it is known as the discrete Gaussian free field. In this case the Hamiltonian is governed by the surface area of the interface. When $\kappa_1=0$, the model is called the membrane, or Bilaplacian, model. In this case the Hamiltonian is governed by the curvature of the interface. More generally the Hamiltonian is governed by an interplay of the surface area and the curvature, hence one considers the model with both gradient and Laplacian interaction. The main aim of this article is to show how the dependency on {the size of the set $\Lambda$} of $\kappa_1$ and $\kappa_2$ affects the scaling limit of $\prob_{\Lambda}$. 

 When $\kappa_1=0$ or $\kappa_2=0$, the scaling limit of the model is well-understood. The literature on the discrete Gaussian free field is huge due to its connection to various other probabilistic objects and we refer the interested reader  to the lecture notes and survey articles~\cite{bere, biskup2017extrema, sheffield2007gaussian}. We refer to \cite{mm_scaling, CarJDScaling,HryVel} for the scaling limit of the membrane model in $d\ge 1$. The literature on the case when $\kappa_1>0, \kappa_2>0$ is limited and has been considered in the works of \cite{sakagawa2018,borecki2010, CarBor,mixed_scaling}. \cite{borecki2010} and \cite{CarBor} introduced this model as the $(\nabla+\Delta)$-model (we will also refer to it as ``mixed model'') with constant $\kappa_1,\,\kappa_2$. They studied in $d=1$ the influence of pinning in order to understand the localization behavior of the polymer. The results were extended to higher dimensions, together with further properties of the free energy, in~\cite{sakagawa2018}. In \cite{mixed_scaling} the scaling limit of the $(\nabla+\Delta)$-model is studied. There it is shown that if one lets the lattice size go to zero, under a suitable scaling the Laplacian term is dominated by the gradient and the limit becomes the Gaussian free field. A very natural question, which we aim at investigating in this paper, is whether one can interpolate between the continuum Gaussian free field and the membrane model by tuning $\kappa_2/\kappa_1$ suitably. To the best of our knowledge, the influence of the length on the shape of the polymer through $\kappa_1$ and $\kappa_2$ has not been systematically addressed in the literature. In~\cite{ruiz2005phase} a phase transition on the surface tension for mixed polymers has been investigated according to a suitable rescaling of $\sqrt{\kappa_2/\kappa_1}$ depending on the lattice size. However the model studied in~\cite{ruiz2005phase} is integer-valued, so it differs from the one studied in the present paper. 
 
We now briefly describe the phase transition picture which appears in the scaling limit. We restrict our focus to $d=1$ for heuristic explanations. Let us consider the Hamiltonian described in~\eqref{eq:Ham_def}. We take {$\Lambda = \{1,\, \ldots,\, N-1\}$ for $N\in \N, \,\kappa_1=1/4$ and $\kappa_2= \kappa(N)/2$}. In $d=1$ in the DGFF case ($\kappa_2=0$) it is well-known that the finite volume measure can be given by a random walk bridge and in the membrane case ($\kappa_1=0$) by an integrated random walk bridge~(\cite{CaravennaDeuschel_pin}). Therefore the scaling limit for the DGFF and membrane turns out to be Brownian bridge and the integrated Brownian bridge, respectively. In $d=1$, a representation for the $(\nabla+\Delta)$-model using random walks was obtained in~\cite{borecki2010}. The details of the representation are recalled in Appendix~\ref{appendix:C}.

  Let $\gamma$ and $\sigma$ be as in \eqref{eq:gamma} and \eqref{eq:sigma}, respectively. Let $(\widetilde \eps_i)_{i\in \Z^{+}}$ be i.i.d. normal random variables with mean zero and variance $\sigma^2/(1-\gamma)^2$. {For $n\ge 1$}, let $W_n= S_n-U_n$,
where $S_n=\sum_{k=1}^n \widetilde \eps_k$ and
$U_n= \gamma^n \widetilde \eps_1+\gamma^{n-1} \widetilde \eps_2+\cdots+ \gamma  \widetilde \eps_n$.
From~\citet[Proposition 1.10]{borecki2010}  it is known that the finite volume measure of the model is given by the joint distribution of {$(W_n)_{1 \le n\le N-1}$} conditioned on $W_{N}=W_{N+1}=0$. We look at the unconditional process and see how the parameter $\kappa(N)$ changes the variance. It follows from~\eqref{eq:gamma} and \eqref{eq:sigma} that {
$$\sigma^2\approx \frac{1}{\kappa(N)}\; \text{ and } \;(1-\gamma) \approx \frac{1}{\sqrt{\kappa(N)}}.$$}
So for the case when $\kappa(N)\ll N^2$ we have {
$$\var(S_{N-1})\approx N ,\; \var(U_{N-1})\approx \sqrt{\kappa(N)}\, \text{ and } \,\cov(S_{N-1}, U_{N-1})\approx \sqrt{\kappa(N)}$$
which together imply that $\var(W_{N-1}) \approx N$}, thus the random walk dominates with its scaling $\sqrt{N}$.

When $\kappa(N)\gg N^2$ the situation is a bit more complicated and one can compute that (see Appendix~\ref{appendix:C})
$$\var(W_{N-1})\approx \frac{N^3}{\kappa(N)}.$$
It turns out that the Laplacian part dominates under this scaling. When $\kappa(N)\sim N^2$ then the contribution from $S_{N-1}$ and $U_{N-1}$ is similar and hence both the gradient and Laplacian interaction come into picture. The reader can see a simulation of the free boundary case, that is, the trajectories of {$(W_n)_{1 \le n\le N}$}, in Figure~\ref{fig1} and Figure~\ref{fig2}. We plotted the two cases $\kappa\ll N^2$ and $\kappa\gg N^2$ in different pictures as the height scalings are different. 

\begin{figure}[ht!]
\includegraphics[scale=0.7]{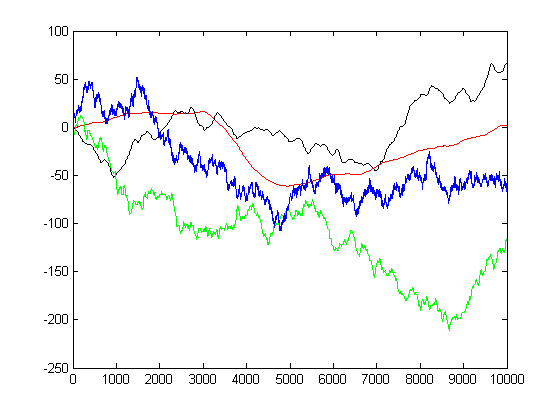}\caption{Simulation of some trajectories of {$(W_n)_{1 \le n\le N}$} with $N=10^4$ and {\color{blue}$\kappa=0$}, {\color{green}$\kappa=2\times 10^2$}, {\color{black}$\kappa=2\times 10^4$}, {\color{red}$\kappa=2\times 10^6$}.}
\label{fig1}
\end{figure}

\begin{figure}[ht!]
\includegraphics[scale=0.7]{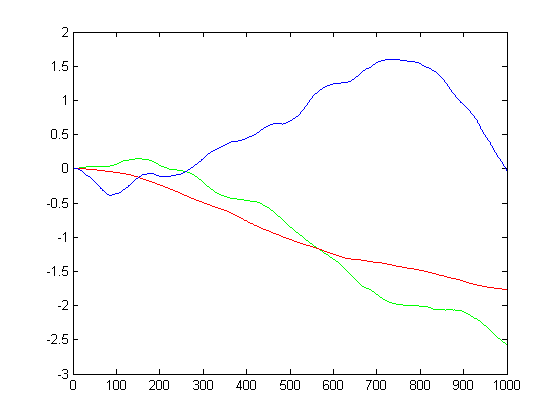}\caption{Simulation of some trajectories of {$(W_n)_{1 \le n\le N}$} with $N=10^3$ and {\color{blue}$\kappa=2\times 10^{6.5}$}, {\color{green}$\kappa=2\times 10^7$}, {\color{red}$\kappa=2\times 10^8$}.}
\label{fig2}
\end{figure}

We stress that in the above description we did not consider boundary effects which can cause considerable difficulty in understanding these processes explicitly. In Appendix~\ref{appendix:C} we have pointed out the conditional representation of $W_{N-1}$. One can see that it is not easy to determine whether the above transition can be pushed to the conditional processes and hence the finite volume measure. The aim of this article is to go beyond such representations and show the above transition holds true in general dimensions and get the explicit limits in each of the cases. In this respect, we also record that the integrated random walk representations of $d=1$ cannot be extended to $d> 1$. We mainly use finite difference methods in the proof of the main results.  In a recent work, the authors of the present article introduced a finite difference method to approximate solutions of PDEs to successfully obtain the scaling limit of the membrane model and the $(\nabla+\Delta)$-model with fixed coefficients (see~\cite{mm_scaling,mixed_scaling}). The idea was inspired by the work~\cite{thomee}. 
 Finite difference methods were also employed in the works~\cite{Mueller:Sch:2017, schweiger:2019} to obtain important estimates on the discrete Green's function of the membrane model.  

The main results of the article are as follows. We consider the model on $\Lambda_N \Subset \Z^d $ for a suitable $\Lambda_N$ defined later in Section \ref{section:main results}. Also, we assume $\kappa_1=1/(4d),\,\kappa_2=\kappa(N)/2$ and distinguish three regimes for $\kappa = \kappa(N)$.
\begin{itemize}
\item[(a)] Let $\kappa\gg N^2$. In $d\ge 1$, we show that the appropriately rescaled field converges to the continuum membrane model. The continuum membrane model is roughly a centered Gaussian process whose covariance is given by the Green's function of the Bilaplacian Dirichlet problem. For $d\ge 4$, in Theorem~\ref{thm:main} we show the convergence takes place in a distributional space (more precisely a negatively-indexed Sobolev space). In $d=1,2$ and $3$ we show in Theorem~\ref{thm:main2} that the limiting Gaussian process has continuous paths.

\item[(b)]  Let $\kappa\sim 2dN^2$. In $d\ge 4$ we show (Theorem~\ref{thm:main}) that the rescaled field converges to a random distribution in an appropriate Sobolev space and the covariance of the limiting Gaussian field is given by the Dirichlet problem involving the elliptic operator $-\Delta_c +\Delta_c^2$. In $d=1,2$ and $3$, again we show (in Theorem~\ref{thm:main2}) the convergence takes place in the space of continuous functions. 

\item[(c)] Let $\kappa\ll N^{2}$. In $d\ge 2$ we show (in Theorem~\ref{thm:main}) that the rescaled field converges in distribution to the Gaussian free field. Again, since the Gaussian free field is a random distribution the convergence takes place in a negatively-indexed Sobolev space. In $d=1$, we show (in Theorem~\ref{thm:main2}) that the limiting process is the Brownian bridge, confirming the heuristics presented above.
\end{itemize}

To derive the above results, the main technique we use is the approximation of the solution of a continuum Dirichlet problem with its discrete counterpart. Using Sobolev estimates it can be shown that the closeness of the solutions is related to the approximation of the discrete elliptic operator to the continuum one. This idea has been already employed in~\cite{mm_scaling} and~\cite{mixed_scaling}.

But in the present scenario, the discrete elliptic operators have coefficients which depend on $N$ and hence the estimates of~\cite{thomee} are not applicable directly. In addition, the rough behavior around the boundary in the case of constant coefficients was dealt with by considering a truncation of the discrete elliptic operator. The operators were rescaled around the boundary and this helped in controlling their behavior. The same technique becomes a bit more involved in the present case. This helps us to tackle with the cases $\kappa\gg N^2$ and $\kappa\sim 2d N^2$ but the method falls short when $\kappa\ll N^2$. In this case an anonymous referee pointed out to the authors the idea of dealing with the boundary effects and discretization separately, adjusting the boundary values with an appropriate cut-off function. We deal with these technical issues in Section~\ref{main_ingredient}. Let us mention in passing that we believe that the result in Section~\ref{main_ingredient} is of independent interest and can be applied to discrete elliptic operators where coefficients depend on the scaling of the lattice.

\paragraph{{\em Structure of the article}} In Section~\ref{section:main results} we state our main results precisely. Furthermore, in its Subsection~\ref{main_ingredient} we discuss the approximation technique and the norm estimates in detail, while in Subsection~\ref{open_problems} we mention some open problems. In Section~\ref{sec:main_proofs} we derive the proof of Theorem~\ref{thm:main} and in Section~\ref{sec:main2_proofs} we deal with the lower dimensional case (Theorem~\ref{thm:main2}). In Section~\ref{proof_approx_result} we provide a proof of the approximation results stated in~Subsection~\ref{main_ingredient}. These are mainly improvements of the results of~\cite{thomee}. 

\paragraph{{\em Notation}} {For real-valued functions $f(\cdot),\,g(\cdot)$ we write $f\gg g,\, f\sim g,\, f\approx g,\,f\ll g$ when $\lim_{n\to \infty} \frac{f(n)}{g(n)}$ equals $\infty,\,1,\,c$ and $0$, respectively, where $c$ is a non zero constant which may be 1 also}. Also we write $f\asymp g$ if there exist two positive constants $c_\ell, c_r$ such that $c_\ell g(n) \le f(n) \le c_r g(n)$ for all $n$. We denote by $C$ a universal constant that may change from line to line within the same equation. In what follows, we shall use $\Delta$ and $\Delta_c$ to denote the discrete and continuous Laplacian respectively. Also $\partial_j$ respectively $\frac{\partial}{\partial x_j}$ denotes the discrete respectively continuous derivative in the $j$-th coordinate. 

\section{\bf Set-up and main results} \label{section:main results}
Let $\Lambda$ be a finite subset of $\Z^d$, $d\ge 1$, and $\prob_{\Lambda}$ and $H(\vr)$ be as in~\eqref{eq:measure_def} and \eqref{eq:Ham_def} respectively.  It follows from Lemma 1.2.2 of~\cite{Kurt_thesis} that the Gibbs measure~\eqref{eq:measure_def} on $\R^{\Lambda}$ with Hamiltonian~\eqref{eq:Ham_def} exists. Note that~\eqref{eq:Ham_def} can be written as 
 \begin{equation}\label{eq:ham:mixed}
H(\vr)=\frac{1}{2}\langle\vr,(-4d\kappa_1\Delta+2\kappa_2\Delta^{2})\vr\rangle_{\ell^{2}(\mathbb{Z}^{d})}.
\end{equation}
Let $d\ge 1$. Let $D$ be a bounded domain in $\R^d$. For $N\in \mathbb N$, let $D_N= N\overline D \cap \mathbb{Z}^d$. Let us denote by $\Lambda_N$ the set of points $x$ in $D_N$ such that, for every  direction $i, j,$ also the points $\,x\pm e_i ,\,x\pm (e_i\pm e_j)$ are all in $D_N$. In other words, $\Lambda_N\subset N\overline D\cap \Z^d$ is the largest set satisfying $\partial_2\Lambda_N\subset N\overline D\cap \Z^d$ where $\partial_2\Lambda_N:=\{y\in\Z^d\setminus\Lambda_N:\mathrm{dist}(y,\,\Lambda_N)\le 2\}$ is the double (outer) boundary of $\Lambda_N$ of points at $\ell^1$ distance at most $2$ from it. We consider the model with $\Lambda = \Lambda_N,\,\kappa_1=1/4d,\,\kappa_2=\kappa(N)/2$ and want to study what happens when we tune suitably the parameter $\kappa(N)$ as $N$ tends to infinity. We assume $\kappa_1$ to be constant as it is easy to state the results in this format. Also for simplicity we write $\kappa$ for $\kappa(N)$. We just note here that if we write $G_{\Lambda_N}(x,\,y):=\E_{\Lambda_N}(\vr_x\vr_y)$, it follows from Lemma 1.2.2 of \cite{Kurt_thesis} that $G_{\Lambda_N}$ solves the following discrete boundary value problem:  for $x\in \Lambda_N$
\begin{equation}\label{eq:cov}
\left\{\begin{array}{lr}
(-\Delta+ \kappa\Delta^2) G_{\Lambda_N}(x,y) = \delta_x(y)& y\in \Lambda_N\\
G_{\Lambda_N}(x,y) = 0  & y\notin \Lambda_N\end{array}.\right.
\end{equation}

To describe the main results we need some elliptic operators. We first introduce them and the corresponding Dirichlet problem.
Let $L$ denote one of the following three elliptic operators: 
\begin{equation}\label{def:L}
L=\begin{cases}-\Delta_c,\\ \Delta_c^2,\\ -\Delta_c +\Delta_c^2,\end{cases}\end{equation}
where $\Delta_c$ is the Laplace operator defined by $\Delta_c=\sum_{i=1}^d\frac{\partial^2}{\partial x^2_i}$.
We consider the following continuum Dirichlet problem:
\begin{equation}\label{eqa:continuum}
 	\begin{cases}
 	Lu(x) = f(x)& x\in D\\
 	D^\alpha u(x)=0& \lvert\alpha\rvert\leq m-1,\,x\in \partial D. 
 	\end{cases}
 	\end{equation}
 	where $\alpha=(\alpha,\ldots,\alpha_d)$ is a multi-index with $\alpha_i$'s being non-negative integers, $|\alpha|:=\sum_{i=1}^d \alpha_i$, $D^\alpha$ is defined in~\eqref{eq:def_D^alpha}, $m=1$ if $L=-\Delta_c$ and $m=2$ in the other cases.
 
\subsection{Lower dimensional results} 	
 We first present the results in lower dimensions where we show that convergence takes place in the space of continuous functions. In this case we consider $D=(0,1)^d$. Also here, according to the behavior of $\kappa$ as $N\to\infty$ we have three different limits. To verify the convergence in the space of continuous functions we shall need to continuously interpolate the discrete model. In $d=1$ the linear interpolation gives a continuous process but for higher dimensions there might be many ways. We stick to the following natural way. We will need this interpolation in $d=2$ and $3$ when $\kappa\gg N^2$ or $\kappa\sim 2dN^2$.
 	We define the continuous interpolation $\{\Psi_N\}_{N\in\N}$ in the following fashion:
\begin{itemize} 
\item For $d=1$ and $t\in\overline D$
\begin{align}
\Psi_N(t)=\mathbf{c}_N(1)\left[\vr_{\floor{Nt}} + (Nt-\floor{Nt})(\vr_{\floor{Nt}+1}-\vr_{\floor{Nt}})\right].\label{eq:intd=1}
\end{align}
\item For $d=2$ and $t=(t_1,t_2)\in \overline D$
\begin{align}
\Psi_{N}(t)&=\mathbf{c}_N(2)\left[\vr_{\floor {Nt}}+\{Nt_i\}\left(\vr_{\floor {Nt}+e_i}-\vr_{\floor {Nt}}\right)\right. \nonumber\\
&+\left.\{Nt_j\}\left(\vr_{\floor {Nt}+e_i+e_j}-\vr_{\floor {Nt}+e_i}\right)\right] ,\,\quad\text{if}\,\{Nt_i\}\ge\{Nt_j\}\label{eq:intd=2}
\end{align}
where $i,\,j\in\{1,\,2\}$, $i\neq j$.
\item For $d=3$ and $t=(t_1,t_2,t_3)\in \overline D$
\begin{align}
\Psi_{N}(t)&=\mathbf{c}_N(3)\left[\vr_{\floor {Nt}}+\{Nt_i\}\left(\vr_{\floor {Nt}+e_i}-\vr_{\floor {Nt}}\right)\right.\nonumber \\
&+\{Nt_j\}\left(\vr_{\floor {Nt}+e_i+e_j}-\vr_{\floor {Nt}+e_i}\right)\nonumber\\
&+\left.\{Nt_k\}\left(\vr_{\floor {Nt}+e_i+e_j+e_k}-\vr_{\floor {Nt}+e_i+e_j}\right)\right],\quad\,\text{ if }\{Nt_i\}\ge\{Nt_j\}\ge\{Nt_k\} \label{eq:intd=3}
\end{align}
where $i,\,j,\,k\in \{1,\,2,\,3\}$ and pairwise different. Here $(e_i)_{i=1}^d$ denotes the standard basis for $\R^d$ and $c_N(d),\,d=1,\,2,\,3$, are scaling factors which are specified in the following result.
\end{itemize}

\begin{theorem}\label{thm:main2} We have the following convergence results.
\begin{enumerate}[ref=(\arabic*)]

\item\label{thm:MM2} $\kappa \gg N^2$. Let $1 \le d \le 3$. Define a continuously interpolated field $\Psi_N$ as in \eqref{eq:intd=1}, \eqref{eq:intd=2} and \eqref{eq:intd=3} with $$\mathbf{c}_N(d)= (2d)^{-1}\sqrt{\kappa}N^{\frac{d-4}{2}}.$$
Then we have, as $N\to\infty$, that the field $\Psi_N$ converges in 
distribution to $\Psi^{\Delta^2}$ in the space of continuous functions on $\overline D$, where $\Psi^{\Delta^2}$ is defined to be the centered continuous Gaussian process on $\overline D$ with covariance $G_D(\cdot,\,\cdot)$, the Green's function for the following biharmonic Dirichlet problem:
\begin{align}
\begin{cases}
\Delta_c^2 u(x) = f(x), & x\in D\\
D^\alpha u(x)=0, &\forall \,|\alpha|\le 1,\, x\in\partial D. \label{eq:mm_continuum}
\end{cases}
\end{align}

\item \label{thm:MIXED2} $\kappa \sim 2d N^2$. Let $1 \le d \le 3$.  Define a continuously interpolated field $\Psi_N$ as in \eqref{eq:intd=1}, \eqref{eq:intd=2} and \eqref{eq:intd=3} with $$\mathbf{c}_N(d)= (2d)^{-1}\sqrt{\kappa}N^{\frac{d-4}{2}}.$$ 
Define $\Psi^{-\Delta + \Delta^2}$ to be the continuous Gaussian process in $\overline D$ with covariance $G_D(\cdot,\,\cdot)$, where $G_D$ is the Green's function for the problem 
\begin{align*}
\begin{cases}
(-\Delta_c + \Delta_c^2) u(x) = f(x), & x\in D\\
D^\alpha u(x)=0, &\forall \,|\alpha|\le 1,\, x\in\partial D. \label{eq:mixed_continuum}
\end{cases}
\end{align*} 
Then $\Psi_N$ converges in distribution to the field $\Psi^{-\Delta + \Delta^2}$ in the space of continuous functions on $\overline D$.


\item \label{thm:GFF2} $\kappa \ll N^{2}$. Let $d=1$. Define the continuously interpolated field $\Psi_N$ as in \eqref{eq:intd=1} with $$\mathbf{c}_N(1)=(2d)^{-\frac12}N^{-\frac12}.$$
Then as $N\to\infty$, $\Psi_N$ converges in distribution to the Brownian bridge, $\Psi^{-\Delta}$, in the space of continuous functions on $\overline D$.
\end{enumerate}
\end{theorem}

\begin{remark}
When $\kappa_1=0$ and $\kappa_2=1$ in \eqref{eq:Ham_def} the $d=1$ case was first studied in~\cite{CarJDScaling}, where they showed that the limiting distribution is given by an integrated Brownian bridge (for a more precise definition see Theorem 1.2 of~\cite{CarJDScaling}). The higher dimensional case was studied in~\cite{mm_scaling}. It was shown in~\cite{mm_scaling} that for $d=2,\,3$ the discrete membrane model converges to a Gaussian process with continuous paths and the methods in that article can be seen to be valid in $d=1$ also. By uniqueness of the limit in $C[0,1]$ it follows that the limiting Gaussian process in $d=1$ for the case $\kappa\gg N^2$ (Theorem~\ref{thm:main2}~\ref{thm:MM2}) can be described using the integrated Brownian bridge, the limit matching that of~\cite{CarJDScaling}.
\end{remark}

\subsection{Higher dimensional results}	
 We  present now the results in higher dimensions where we show convergence in the space of distributions. In order to make our statements precise, we need to introduce three (negative ordered) Sobolev spaces denoted respectively as $\mathcal H^{-s}_{\Delta^2}(D)$, $\mathcal H^{-s}_{-\Delta+\Delta^2}(D)$ and $\mathcal H^{-s}_{-\Delta}(D)$\footnote{We shall use $\Delta$ in the subscript of the spaces and the norms instead of $\Delta_c$ to ease notation.}. We are going to recall some basic notations on Sobolev spaces and also some facts about the eigenvalues of the elliptic operators involved in our problem.
 
 \subsubsection{Basics of Sobolev spaces}
 Let us first describe the standard Sobolev space. Let $C_c^\infty(D)$ denote the space of 
infinitely differentiable functions $u: D\to \R$ with compact support inside 
$D$. For $\alpha= (\alpha_1, \,\ldots,\, \alpha_d)$ a multi-index define
\begin{equation}\label{eq:def_D^alpha}
D^\alpha u= \frac{\partial^{\alpha_1}}{\partial x_1^{\alpha_1}}\cdots \frac{\partial^{\alpha_d}}{\partial x_d^{\alpha_d}} u.
\end{equation}
Suppose $f, \,g\in L^1_{loc}(D)$. We say that $g$ is the $\alpha$-th weak partial derivative of 
$f$ (written $D^\alpha f= g$) if 
$$\int_{D} f D^\alpha u \De x= (-1)^{|\alpha|} \int_{D} g u \De x \quad\forall\, u \in C_c^\infty(D).$$
The Sobolev space $W^{k,p}$ is defined in the usual way as
$$W^{k,p}= \{ f\in L^1_{loc}(D) : \,D^\alpha f\in L^p(D), \, |\alpha|\le k\}.$$
Denote by $H^k(D):= W^{k,2}(D)$, $k=0,\,1,\,\ldots$, which is a Hilbert space with norm
$$\|f\|_{H^k(D)} = \left( \sum_{|\alpha|\le k}\int_{D} |D^\alpha f|^2\De x\right)^{1/2}.$$
It is true that if $a>b$ then $H^a(D)\subset H^b(D)$. Let us define another Hilbert space,
$$H^k_0(D):= \overline{ C_c^{\infty}(D)}^{\|\cdot\|_{H^k(D)}}$$
and let $H^{-k}(D)= [ H^k_0(D)]^*$ be its dual.

\subsubsection{Continuum membrane model}
We briefly give the definition of the Sobolev space $\mathcal H^{-s}_{\Delta^2}(D)$ and the continuum membrane model. For a more detailed discussion see \cite{mm_scaling}. By the spectral theorem for compact self-adjoint operators and elliptic regularity one can show that there exist smooth eigenfunctions $\{u_j\}_{j\in \N}$ of $\Delta_c^2$ corresponding to the eigenvalues $0<\lambda_1\le\lambda_2\le \cdots \to\infty$ such that $\{u_j\}_{j\in \N}$ is an orthonormal basis for $L^2(D)$. Now for any $s>0$ we define the following inner product on $C_c^\infty(D)$:
$$\la f\,,\,g\ra_{s,\,\Delta^2}:=\sum_{j\in \N}\lambda_j^{s/2}\la f\,,\,u_j\ra_{L^2}\la u_j\,,\,g\ra_{L^2}.$$
Then $\mathcal H^{s}_{\Delta^2, 0}(D)$ is defined to be the Hilbert space completion of $C_c^\infty(D)$ with respect to this inner product. We define $\mathcal H^{-s}_{\Delta^2}(D)$ to be its dual and the dual norm is denoted by $\| \cdot \|_{-s,\,\Delta^2}$. The following definition is from \citet[Proposition~3.9]{mm_scaling} and provides a description of the continuum membrane model $\Psi^{\Delta^2}$.

\begin{definition}\label{prop:series_rep_h}
 		Let $(\xi_j)_{j\in \N}$ be a collection of i.i.d. standard Gaussian random variables. Set
 		\[
 		\Psi^{\Delta^2}:=\sum_{j\in \N}\lambda_j^{-1/2}\xi_j u_j.
 		\]
 		Then $\Psi^{\Delta^2}\in \mathcal H^{-s}_{\Delta^2}(D)$ a.s. for all $s>({d-4})/2$ and is called the continuum membrane model.
 	\end{definition}
 	
\subsubsection{Continuum mixed model} We define the space $\mathcal H^{-s}_{-\Delta+\Delta^2}(D)$ analogously to $\mathcal H^{-s}_{\Delta^2}(D)$. One can find smooth eigenfunctions $\{v_j\}_{j\in\mathbb N}$ of $-\Delta_c+\Delta^2_c$ corresponding to eigenvalues $0<\mu_1\le \mu_2\le \cdots\to \infty$ such that $\{v_j\}_{j\in \mathbb N}$ is an orthonormal basis of $L^2(D)$. One can define, for $s>0$, the following inner product for functions from $C_c^\infty(D)$:
$$\la f, g\ra_{s,\,-\Delta + \Delta^2}:=\sum_{j\in \mathbb N} \mu_j^{s/2}\la f, v_j\ra_{L^2}\la v_j, g\ra_{L^2}.$$
Let $\mathcal H^{s}_{-\Delta+\Delta^2, 0}(D)$ be the completion of $C_c^\infty(D)$ with the above inner product and $\mathcal H^{-s}_{-\Delta+\Delta^2}(D)$ be its dual. The dual norm is denoted by $\|\cdot\|_{-s,\,-\Delta + \Delta^2}$. We describe the details on this space in Appendix~\ref{appendix:B}. The following definition is proved as Proposition~\ref{prop:mixed_series_rep_h} in Appendix~\ref{appendix:B}.

\begin{definition}
	Let $(\xi_j)_{j\in \N}$ be a collection of i.i.d. standard Gaussian random variables. Set
	\[
	\Psi^{-\Delta+\Delta^2}:=\sum_{j\in \N}\mu_j^{-1/2}\xi_j v_j.
	\]
	Then $\Psi^{-\Delta+\Delta^2}\in \mathcal H^{-s}_{-\Delta+\Delta^2}(D)$ a.s. for all $s>({d-4})/2$ and is called the continuum mixed model.
\end{definition}
 	
\subsubsection{Gaussian free field} 	
 Here also we briefly give the definition of the Sobolev space $\mathcal H^{-s}_{-\Delta}(D)$ and the Gaussian free field. For a detail discussion see \cite{mixed_scaling}. By the spectral theorem for compact self-adjoint operators and elliptic regularity we know that there exist smooth eigenfunctions $(w_j)_{j\in\N}$ of $-\Delta_c$ corresponding to the eigenvalues $0<\nu_1\le\nu_2\le\cdots \to\infty$ such that $(w_j)_{j\ge 1}$ is an orthonormal basis of $L^2(D)$. Now for any $s>0$ we define the following inner product on $C_c^\infty(D)$:
$$\langle f\,,\,g \rangle_{s,\,-\Delta}:=\sum_{j\in\N} \nu_j^s\langle f\,,\,w_j\rangle_{L^2}\langle w_j\,,\,g\rangle_{L^2} .$$ 
Then $\mathcal H^s_{-\Delta, 0}(D)$ can be defined to be the completion of $C_c^\infty(D)$ with respect to this inner product. We define $\mathcal H^{-s}_{-\Delta}(D)$ to be its dual and the dual norm is denoted by $\| \cdot \|_{-s,\,-\Delta}$. We give the definition of the Gaussian free field in the next Proposition.

\begin{definition}[{\citet[Proposition~10]{mixed_scaling}}]\label{prop:series_rep_gff}
	Let $(\xi_j)_{j\in \N}$ be a collection of i.i.d. standard Gaussian random variables. Set 
	\[
	\Psi^{-\Delta}:=\sum_{j\in \N}\nu_j^{-1/2}\xi_j w_j.
	\]
	Then $\Psi^{-\Delta}\in \mathcal H^{-s}_{-\Delta}(D)$ a.s. for all $s>{d}/2 -1$ and is called the Gaussian free field.
\end{definition}

\begin{remark}
    We define different spaces with respect to different eigenfunctions of the operators. It is not clear to us if these spaces coincide for a general domain. We are not aware of a result which gives the norm equivalence between the spaces $\mathcal H^s_{\Delta^2, 0}(D)$, $\mathcal H^s_{-\Delta+\Delta^2, 0}(D)$ and $\mathcal H^s_{-\Delta, 0}(D)$. In this article we are not pursuing this line of research; what is important for us are the specific norms that determine the limiting variance of the discrete fields.
\end{remark}

\begin{remark}
    Note that we have used the same notation for the fields both in higher as well as as in lower dimensions, although they do not live in the same spaces. The relation of the fields comes through the Dirichlet problem. For $f\in C_c^\infty(D)$, one can easily show that
    $$\E[(\Psi^{L}, f)^2]=\iint_{D\times D} G_L(x,y) f(x)f(y)\De x\De y$$
    where $\Psi^{L}$ is one of the three fields associated to the elliptic operator $L$ as in \eqref{def:L} and $G_L$ is the Green's function of the Dirichlet problem~\eqref{eqa:continuum}.
\end{remark}

 We are now ready to state our main results in the higher dimensional case.


\begin{theorem}\label{thm:main}
Assume that $D$ has smooth boundary. Depending on the behavior of $\kappa$ as $N\to\infty$ we have the following three convergence results. 
\begin{enumerate}[ref=(\arabic*)]

\item\label{thm:MM} $\kappa \gg N^2$. Let $d\ge 4$. Define $\Psi_N$ by 
\begin{align}(\Psi_N,\,f):=(2d)^{-1}\sqrt{\kappa}N^{-\frac{d+4}2}\sum_{x\in \frac1N\Lambda_N }\vr_{Nx}f(x) ,\quad f\in \mathcal H^s_{\Delta^2, 0}(D).\label{eq:mm_psi_N}
\end{align}
Then we have, as $N\to \infty$, that the field $\Psi_N$ converges in distribution to the continuum membrane model $\Psi^{\Delta^2}$ in the topology of $\mathcal H^{-s}_{\Delta^2}(D)$ for $s>s_d$, where 
\begin{equation}\label{eq:sd}
s_d:=\frac{d}{2} + 2\left(\ceil*{ \frac{1}{4}\left(\floor*{\frac{d}{2}}+1\right)} + \ceil*{ \frac{1}{4}\left(\floor*{\frac{d}{2}}+6\right)} -1\right).
\end{equation}

\item \label{thm:MIXED} $\kappa \sim 2d N^2$. Let $d\ge 4$. Define $\Psi_N$ by 
\begin{align}(\Psi_N,\,f):=(2d)^{-1}\sqrt{\kappa}N^{-\frac{d+4}2}\sum_{x\in \frac1N\Lambda_N }\vr_{Nx}f(x) ,\quad f\in \mathcal H^s_{-\Delta+\Delta^2, 0}(D).\label{eq:mixed_psi_N}
\end{align}
Then, as $N\to \infty$, the field $\Psi_N$ converges in distribution to $\Psi^{-\Delta+\Delta^2}$ in the topology of $\mathcal H^{-s}_{-\Delta+\Delta^2}(D)$ for $s>s_d$ where $s_d$ is as in~\eqref{eq:sd}.

\item \label{thm:GFF} $\kappa \ll N^{2}$. Let $d\ge 2$. Define $\Psi_N$ by 
\begin{align}(\Psi_N,\,f):=(2d)^{-\frac12}N^{-\frac{d+2}2}\sum_{x\in \frac1N\Lambda_N }\vr_{Nx}f(x) ,\quad f\in \mathcal H^s_{-\Delta, 0}(D).\label{eq:gff_psi_N}
\end{align}
Then, as $N\to \infty$, the field $\Psi_N$ converges in distribution to the Gaussian free field $\Psi^{-\Delta}$ in the topology of $\mathcal H^{-s}_{-\Delta}(D)$ for $s>{d}/2 + \floor{{d}/2}+2$.
\end{enumerate}
\end{theorem}

\begin{remark}
Note that the convergence takes place in a larger Sobolev space than where the field is defined. The appearance of $s_d$ in~\eqref{eq:sd} is due to the tightness proof. We believe that sharp results on convergence, in particular on the index $s_d$, could be obtained with other methods. However we do not pursue optimality results in the present article. 
\end{remark}

\subsection{Main ingredients in the proofs}\label{main_ingredient}
We prove both Theorem~\ref{thm:main2} and Theorem~\ref{thm:main} by first showing finite dimensional convergence and secondly tightness. As the measures are Gaussian with mean zero, the finite dimensional convergence follows from the convergence of the covariance. However the behavior of the covariance of the model is not known explicitly. Therefore we use the expedient of finite difference schemes to achieve both goals. The key fact which allows us to employ PDE techniques is that the covariance satisfies the discrete boundary value problem \eqref{eq:cov}. For the proof of our main theorems we will compute in Theorem~\ref{approx_result} the magnitude of the error one commits in approximating the solution of the Dirichlet problem~\eqref{eqa:continuum} by its discrete counterpart.
In the present section we only state the error estimate leaving the proof for Section \ref{proof_approx_result}. Let $D$ be any bounded domain in $\mathbb{R}^d$ satisfying the uniform exterior ball condition (UEBC), which states that there exists $r>0$ such that for any $z\in\partial D$ there is a ball $B_r(c)$ of radius $r$ with center at some point $c$ satisfying $\overline{B_r(c)}\cap \overline D = \{z\}$. We mention here that any domain with $C^2$ boundary satisfies the UEBC. 

 Let $h>0$. We will call the points in $h\mathbb{Z}^d$ the grid points in $\mathbb{R}^d$. We consider $L_h$ to be a discrete approximation of $L$ given by 
\begin{align}\label{discrete_op}
L_hu=\left\lbrace 
\begin{array}{l l l}
(-\Delta_h+\rho_1(h)\Delta^2_h)u& \quad  \text{if $L=-\Delta_c$}\\
(-\rho_2(h)\Delta_h+\Delta^2_h)u& \quad  \text{if $L=\Delta_c^2$}\\
(-\Delta_h+ \rho_3(h)\Delta^2_h)u& \quad  \text{if $L=-\Delta_c +\Delta_c^2$}
\end{array} \right.
 \end{align}
 where $\Delta_h$ is defined by $$\Delta_hu(x):=\frac1{h^2}\sum_{i=1}^d (u(x+he_i)+u(x-he_i)-2u(x)),$$
 $u$ is any function on $h\mathbb{Z}^d$ (called a grid function) and $\rho_i(h)$ are functions of $h$ taking values in the positive real line such that 
 $$\lim_{h\to 0}\rho_i(h)=\begin{cases} 0 & i=1,\,2\\ 1& i=3\end{cases}.$$
 
 Let $D_h$ be the set of grid points in $\overline D$, i.e. $D_h=\overline D\cap h\mathbb{Z}^d$. For any grid point $x$ we define the points $x\pm he_i,\,x\pm h(e_i\pm e_j)$ with $1\le i,j\le d$ to be its neighbors. We say that $x$ is an interior grid point in $D_h$ if all its neighbors are in $D_h$. Let $R_h$ be the set of interior grid points in $D_h$ and $B_h:= D_h\setminus R_h $ be the set of grid points near the boundary. We divide $R_h$ further into $R^*_h$ and $B^*_h$, where $R^*_h$ is the set of $x$ in $R_h$ such that all its neighbors are in $R_h$ and $B^*_h$ is the set of remaining points in $R_h$. 
 Thus we have $$D_h=B_h\cup R_h=B_h\cup B^*_h\cup R^*_h.$$
Denote by $\mathcal{D}_h$ the set of grid functions vanishing outside $R_h$. 
 	For a grid function $f$ we define $R_hf\in\mathcal{D}_h$ by
 	\begin{equation}
\label{eq:R_h}
 	R_hf(x)=\begin{cases}
 	f(x) &  x\in R_h\\
 	0 & x\notin R_h
 	\end{cases}.
 	\end{equation}
Define for grid-functions vanishing outside a finite set
\begin{align*}
& \la u\,,\,v \ra_{h,\,grid} := h^d \sum_{x\in h\Z^d}u(x)v(x),\\
& \|u\|_{h,\,grid}:= \la u\,,\,u \ra_{h,\,grid}^{1/2}.
\end{align*}
We now define the finite difference analogue of the Dirichlet problem \eqref{eqa:continuum}. For given $h$, we look for a function $u_h(\cdot)$ defined on $D_h$ such that 
 	\begin{align}
 	L_hu_h(x)=f(x), \quad x\in R_h \label{eq:discrete}
 	\end{align} 
and
 	\begin{align}
 	u_h(x)=0 ,\quad x\in B_h. \label{eq:discrete boundary}
 	\end{align}
 The uniqueness of the solution of \eqref{eq:discrete} and \eqref{eq:discrete boundary} is shown in Lemma~\ref{fact:unique_discrete}.
 We are now ready to state the error estimate result which forms the core result of this article. 
 \begin{theorem}\label{approx_result}
 Depending on $L$ we have the following error bounds.
 \begin{enumerate}[ref=(\arabic*)]
 \item\label{thm:mm_one}$L=\Delta_c^2$. Let $u\in {C}^5(\overline{D})$ be the solution of the Dirichlet problem \eqref{eqa:continuum}. If $e_h:=u-u_h$ then we have for all sufficiently small $h$
 		$$\lVert R_he_h\rVert_{h,\,grid}^2\le C\left[M_5^2h^2 + M_2^2(\rho_2(h))^2 +  M_2^2 h\right].$$

\item\label{thm:mixed_one} $L=-\Delta_c + \Delta_c^2$.	Let $u\in {C}^5(\overline{D})$ be the solution of the Dirichlet problem~\eqref{eqa:continuum}. If $e_h:=u-u_h$ then we have for all sufficiently small $h$
\begin{align*}\lVert R_he_h\rVert_{h,\,grid}^2 \le C\left[M_5^2h^2 + M_4^2(\rho_3(h)-1)^2  + M_4^2 h^4 + M_2^2 h \right].\end{align*}

\item\label{thm:gff_one}$L= -\Delta_c$. Let $u\in {C}^4(\overline{D})$ be a solution of the Dirichlet problem~\eqref{eqa:continuum}. If $e_h:=u-u_h$ then for sufficiently small $h$ we have
	 $$\|R_he_h\|_{h,\,grid}^2 \le C \left[M_4^2 \delta^4 + M_2^2 \rho_1(h) \delta + M_1^2 \delta \right],$$
where $\delta:= \max \{ h, \sqrt{\rho_1(h)}\}$.	 
\end{enumerate} 
 In all the cases $M_k:=\sum_{\lvert\alpha\rvert\le k}\sup_{x\in D}\lvert D^\alpha u(x)\rvert$.	
 \end{theorem}

\subsection{Open problems and discussions}\label{open_problems}
In this subsection we list some open problems.

\begin{enumerate}


\item Let $\eps\ge 0$ and consider the following pinned measure on $\R^{V_N}$, with $V_N$ being a box of side length $N$:
$$\prob_{\eps,N}=\frac{1}{Z_{\eps, N}}\mathrm{e}^{-H(\phi)}\prod_{x\in{V_N}}(\eps \delta_0(\mathrm{d}\phi_x)+\mathrm{d}\phi_{x})\prod_{x\in\mathbb{Z}^{d}\setminus V_N}\delta_{0}(\mathrm{d}\phi_{x})$$
Here $H(\phi)$ is as in \eqref{eq:ham:mixed}. Let $F(\eps)$ be the free energy of the above system, namely,
$$F(\eps)= \lim_{N\to\infty} \frac{1}{N} \log\frac{Z_{\eps, N}}{Z_{0,N}}.$$ 
If $F(\eps)>0$ then the above pinned measure is said to be {\it localized}, otherwise it is {\it delocalized}. We call $\eps_c$ the supremum of all delocalized $\eps$. It would be interesting to see if the above model with $\kappa_1$ and $\kappa_2$ depending on $N$ shows a phase transition with respect to localization. The case when $\kappa_1$ and $\kappa_2$ do not depend on $N$ was studied in \cite{CarBor}. The case of $\kappa_1=0$ and $d=1$ was extensively studied in the literature, see~ \cite{CaravennaDeuschel_pin, CarJDScaling}. 

\item Extremes of interface models are also to be investigated. From Theorem~\ref{thm:main2} it follows that the maximum of the $(\nabla+\Delta)-$model with varying coefficients converges after appropriate rescaling to the supremum of a Gaussian process. We summarise the cases in which we are able to identify the limiting rescaled maximum:
\begin{itemize}
\item $\kappa\ll N^{2}$ and $d=1$;
\item $\kappa\sim 2d N^2$ and $d=1,\,2,\,3$;
\item $\kappa\gg N^2$ and $d=1,\,2,\,3$;
\end{itemize}
All the remaining cases are not known yet and it would be interesting to see if the existing methods can be pushed to cover other dimensions. The challenge in this problem arises because the behavior of the Green's function is hard to determine. A similar situation was recently handled by \cite{schweiger:2019} to determine the extremes of the the four-dimensional membrane model. He found out estimates for the Green's function and applied the methods of \cite{ding:roy:ofer} to show that the limit of the maximum is a shifted Gumbel distribution. 

\end{enumerate}

\section{\bf Proof of Theorem~\ref{thm:main}}\label{sec:main_proofs}
We now give the proof of each of the three parts of Theorem~\ref{thm:main}.

%

\subsection{Proof of finite dimensional convergence}
We first show that for $f\in C_c^\infty(D)$ 
\begin{equation}
\label{eq:mm_toshowfdm}
( \Psi_N, f)\overset{d}\longrightarrow \begin{cases}
(\Psi^{\Delta^2}\,,\,f) & \kappa \gg N^2\\
(\Psi^{-\Delta+\Delta^2}, f) & \kappa \sim 2d N^2\\
(\Psi^{-\Delta}, f) & \kappa \ll N^{2}
\end{cases}.\end{equation}
We begin by noting that $(\Psi_N, f)$ is a centered Gaussian random variable. Hence to show the above convergence it is enough to show that $\var(\Psi_N, f)$ converges to the variance of the Gaussian on the right hand side of~\eqref{eq:mm_toshowfdm}.
We denote $G_{\frac1N}(x,y):=\E_{\Lambda_N}[ \vr_{Nx}\vr_{Ny}]$. Note that by ~\eqref{eq:cov}, we have for all $x\in \frac1N\Lambda_N$,
\begin{equation}
\kappa\gg N^2:\qquad \begin{cases}
\left(-\frac{2dN^2}{\kappa}\Delta_{\frac1N}+\Delta_{\frac1N}^2\right) G_{\frac1N}(x,y) =\frac{4d^2N^4}{\kappa}\delta_{x}(y), & y\in \frac1N\Lambda_N \\
G_{\frac1N}(x,y) = 0  & y\notin \frac1N\Lambda_N\label{eq:mm_G_N}
\end{cases}
\end{equation}

\begin{align}
\kappa\sim 2d N^2: \qquad \begin{cases}
		\left(-\Delta_{\frac1N}+\frac{\kappa}{2dN^2}\Delta_{\frac1N}^2\right) G_{\frac1N}(x,y) =2dN^2\delta_{x}(y), & y\in \frac1N\Lambda_N \\
		G_{\frac1N}(x,y) = 0 & y\notin \frac1N\Lambda_N\label{eq:mixed_G_N}
		\end{cases}
\end{align}

\begin{equation}
\kappa\ll N^{2}:\qquad		\begin{cases}
		\left(-\Delta_{\frac1N}+\frac{\kappa}{2dN^2}\Delta_{\frac1N}^2\right) G_{\frac1N}(x,y) =2dN^2\delta_{x}(y), & y\in \frac1N\Lambda_N \\
		G_{\frac1N}(x,y) = 0 & y\notin \frac1N\Lambda_N.
		\end{cases}\label{eq:gff_G_N}
		\end{equation}

Now considering all the three cases we can rewrite the variance as 
$$\var[(\Psi_N, f)]=N^{-d}\sum_{x\in \frac1N\Lambda_N} H_N(x) f(x)$$
where for $x\in \frac1N D_N$,
$$H_N(x)=\begin{cases}
(2d)^{-2}\kappa N^{-4}\sum_{y\in \frac1N\Lambda_N}G_{\frac1N}(x,y) f(y)& \kappa\gg N^2\\
(2d)^{-2}\kappa N^{-4}\sum_{y\in \frac1N\Lambda_N}G_{\frac1N}(x,y) f(y) & \kappa\sim 2d N^2\\
(2d)^{-1}N^{-2}\sum_{y\in \frac1N\Lambda_N}G_{\frac1N}(x,y) f(y) & \kappa\ll N^{2}.
\end{cases}$$
It is immediate from \eqref{eq:mm_G_N}, \eqref{eq:mixed_G_N}, \eqref{eq:gff_G_N} that $H_N$ is the solution of the following Dirichlet problem:
\begin{align}
\kappa\gg N^2: \qquad \begin{cases}
 \left(-\frac{2dN^2}{\kappa}\Delta_{\frac1N}+\Delta_{\frac1N}^2\right) H_N(x) = f(x), & x\in \frac1N\Lambda_N\\
H_N(x)= 0, & x\notin \frac1N\Lambda_N\label{mm_H_N}
\end{cases}
\end{align}

\begin{align}
\kappa\sim 2dN^2: \qquad \begin{cases}
 \left(-\Delta_{\frac1N}+ \frac{\kappa}{2dN^2}\Delta_{\frac1N}^2\right) H_N(x) = f(x) & x\in \frac1N\Lambda_N\\
H_N(x)= 0 & x\notin \frac1N\Lambda_N\label{mixed_H_N}
\end{cases}
\end{align}

\begin{align}
\kappa\ll N^{2}: \qquad \begin{cases}
		\left(-\Delta_{\frac1N}+\frac{\kappa}{2dN^2}\Delta_{\frac1N}^2\right)H_N(x) = f(x), & x\in \frac1N\Lambda_N\\
		H_N(x)= 0, & x\notin \frac1N\Lambda_N.\label{gff_H_N}
		\end{cases}
\end{align}
Observe that we get the discrete Dirichlet problem involving the operator $L_h$ defined in \eqref{discrete_op} with $h=1/N$ and $$\rho_1(h):=\kappa h^2/2d,\quad \rho_2(h):= 2d/\kappa h^2, \quad \rho_3(h):=\kappa h^2/2d.$$
We now recall the continuum Dirichlet problem~\eqref{eqa:continuum} with the elliptic operator $L$ as in~\eqref{def:L}:
\begin{equation*}
\begin{cases}
 	Lu(x) = f(x)& x\in D\\
 	D^\alpha u(x)=0& \lvert\alpha\rvert\leq m-1,\,x\in \partial D.
 	\end{cases}
 	\end{equation*}
 	where $m=1$ if $L=-\Delta_c$ and $m=2$ in the other two cases. We set $L:=\Delta_c^2$ when $\kappa\gg N^2$, $L:=-\Delta_c$ when $\kappa\ll N^{2}$ and $L:=-\Delta_c+\Delta_c^2$ when $\kappa\sim 2dN^2$.
 Define $e_N(x)= H_N(x)-u(x)$ for $x\in \frac1N D_N$. Then from Theorem~\ref{approx_result} we have
\begin{equation}\label{eq:mm_thomee}
N^{-d}\sum_{x\in \frac1N\Lambda_N} e_N(x)^2\le
\begin{cases}
 C\left( \frac1{N^2} +\frac{4d^2N^4}{\kappa^2} + \frac1N\right) & \kappa\gg N^2\\
 C\left(\frac1N + \left(\frac{\kappa}{2dN^2}-1 \right)^2\right) & \kappa\sim 2d N^2\\
 C \max\{\frac{1}{N},\frac{\sqrt{\kappa}}{\sqrt{2d}N}\}  & \kappa\ll N^{2}.
 \end{cases}.
\end{equation}
Hence we get that
\begin{equation}\label{eq:varexp}
\var[(\Psi_N, f)]= \sum_{x\in \frac1N\Lambda_N} e_N(x) f(x)N^{-d} + \sum_{x\in \frac1N\Lambda_N} u(x) f(x) N^{-d}.
\end{equation}
Note that by Cauchy-Schwarz inequality and ~\eqref{eq:mm_thomee} the first term goes to zero as $N\to \infty$. The second term converges to
\begin{equation}\label{eq:mm_limit}
\sum_{x\in \frac1N\Lambda_N} u(x) f(x) N^{-d}\to_{N\to \infty} \int_{D} u(x) f(x)\De x.
\end{equation}
Notice that by integration by parts we have 
$$\int_{D} u(x) f(x)\De x=
\begin{cases}
\|u\|_{2,\,\Delta^2}^2=\|f\|_{-2,\,\Delta^2}^2 & L= \Delta_c^2\\
\|u\|^2_{2,\,-\Delta + \Delta^2} = \|f\|_{-2,\,-\Delta + \Delta^2}^2 &  L= -\Delta_c+\Delta_c^2\\
\|u\|_{1, -\Delta}^2=\|f\|_{-1, -\Delta}^2 &  L=-\Delta_c.
\end{cases}
$$
On the other hand from the definition it follows that
\begin{align*}
\var[(\Psi^{\Delta^2}\,,\,f)]&=\sum_{j\in \N}\lambda_j^{-1} \la u_j\,,\,f\ra_{L^2}^2=\|f\|_{-2,\,\Delta^2}^2\\
\var[(\Psi^{-\Delta+\Delta^2}\,,\,f)]&=\sum_{j\in \N}\mu_j^{-1} \la v_j\,,\,f\ra_{L^2}^2=\|f\|_{-2,\,-\Delta + \Delta^2}^2\\
	\var[(\Psi^{\Delta}\,,\,f)]&=\sum_{j\in \N}\nu_j^{-1} \la w_j\,,\,f\ra_{L^2}^2=\|f\|_{-1,-\Delta}^2.
\end{align*}
Consequently we obtain~\eqref{eq:mm_toshowfdm}.

\subsection{Tightness}

To show tightness we shall need the following bounds on the eigenfunctions $(u_j)_{j\in \mathbb N}$, $(v_j)_{j\in \mathbb N}$ and $(w_j)_{j\in \mathbb N}$ of $\Delta_c^2$, $-\Delta_c+\Delta_c^2$ and $-\Delta_c$ respectively. They can obtained from the general Sobolev inequality (\citet[Chapter~5, Theorem~6 (ii)]{Evans}) and a repeated application of \citet[Corollary~2.21]{GGS}. 

\begin{lemma} Let $$l_k:=\ceil*{ \frac{1}{4}\left(\floor*{\frac d2}+k+1\right)},\quad k\ge 0. $$ 
	\begin{enumerate}
	\item For the eigenfunctions $(u_j)_{j\in \mathbb N}$ of $\Delta_c^2$ in Problem~\eqref{eqa:continuum} there exists a constant $C>0$ such that for $k\ge 0$
	\begin{equation}\label{bound0:MM}
	\sum_{\lvert\alpha\rvert\le k}\sup_{x\in D}\lvert D^\alpha u_j(x)\rvert \le C\lambda_j^{l_k}.
	\end{equation}
	
	\item For the eigenfunctions  $(v_j)_{j\in \mathbb N}$ of $-\Delta_c+\Delta_c^2$ in Problem~\eqref{eqa:continuum} there exists a constant $C>0$ such that for $k\ge 0$ 
	\begin{equation}\label{bound0:mixed}
	\sum_{\lvert\alpha\rvert\le k}\sup_{x\in D}\lvert D^\alpha v_j(x)\rvert \le C\mu_j^{l_k}.
	\end{equation}
	
	\item For the eigenfunctions  $(w_j)_{j\in \mathbb N}$ of $-\Delta_c$ in Problem~\eqref{eqa:continuum} there exists a constant $C>0$ such that for $k\ge 0$
	\begin{equation}\label{bound0:GFF}
	\sum_{\lvert\alpha\rvert\le k}\sup_{x\in D}\lvert D^\alpha w_j(x)\rvert \le C\nu_j^{\frac{\floor{\frac{d}2}+k+1}2}. 
	\end{equation}
\end{enumerate}
In each instance, the constant $C$ may depend on $k$.
\end{lemma}
We can now begin to show tightness.\newline
{\bf Case 1: $\kappa\gg N^2$.}
    Our target is to show that the sequence $(\Psi_N)_{N\in \mathbb{N}}$ is tight in $\mathcal H^{-s}_{\Delta^2}(D)$ for all $s>s_d$. It is enough to show that
\begin{equation}\label{mm_limsup_psi}
\limsup_{N\to \infty}\E_{\Lambda_N}[\|\Psi_N\|_{-s,\,\Delta^2}^2]<\infty \quad \forall \, s>s_d.
\end{equation}
The tightness of $(\Psi_N)_{N\in \mathbb{N}}$ would then follow immediately from \eqref{mm_limsup_psi} and the fact that, for $0\le s_1<s_2$, $\mathcal H^{-s_1}_{\Delta^2}(D)$ is compactly embedded in $\mathcal H^{-s_2}_{\Delta^2}(D)$ (for a proof of this fact see \citet[Theorem~3.15]{mm_scaling}).

From the definition of dual norm it is immediate that we have
$$\E_{\Lambda_N}\left[\|\Psi_N\|_{-s,\,\Delta^2}^2\right] \le \sum_{j\in \N}\lambda_j^{-s/2}\E_{\Lambda_N}[(\Psi_N\,,\,u_j)^2].$$
Note that $u=\lambda_j^{-1}u_j$ is the unique solution of \eqref{eqa:continuum} with $L=\Delta_c^2$ for $f:=u_j$. Define $e_{N,j}$ to be the error between the solution of the discrete Dirichlet problem~\eqref{mm_H_N} and the continuum one~\eqref{eqa:continuum} with input datum $f:=u_j$. Now as in~\eqref{eq:varexp} we have
	\begin{align}
	\E_{\Lambda_N}[(\Psi_N\,,\,u_j)^2]&= \sum_{x\in \frac1N\Lambda_N} e_{N,j}(x) u_j(x)N^{-d} + \sum_{x\in \frac1N\Lambda_N} \lambda_j^{-1}u_j(x) u_j(x) N^{-d}\nonumber\\
	& \le C\sup_{x\in D}|u_j(x)| \left({N^{-d} \sum_{x\in \frac1N\Lambda_N} e_{N,j}(x)^2}\right)^{1/2}+ C\lambda_j^{-1}\left(\sup_{x\in D}|u_j(x)|\right)^2.\label{eq:varbound}
	\end{align}
	Using Theorem \ref{approx_result} \ref{thm:mm_one} along with the bounds~\eqref{bound0:MM} we obtain
	\begin{align*}
	\E_{\Lambda_N}[(\Psi_N\,,\,u_j)^2]&\le C \lambda_j^{l_0} \left[\lambda_j^{2l_5-2}N^{-2} + \lambda_j^{2l_2-2} 4d^2N^4\kappa^{-2}+\lambda_j^{2l_2-2} N^{-1}\right]^{\frac12} + C\lambda_j^{2l_0-1}\\
	&\le C \lambda_j^{l_0 + l_5-1}.
	\end{align*}
	Therefore we have 
	\begin{align*}
	\E_{\Lambda_N}\left[\|\Psi_N\|_{-s,\,\Delta^2}^2\right]\le C \sum_{j\in \N}\lambda_j^{-\frac s2}\lambda_j^{l_0 + l_5-1}.
	\end{align*}
	Thus 
	$$\limsup_{N\to \infty}\E_{\Lambda_N}[\|\Psi_N\|_{-s,\,\Delta^2}^2]<\infty \qquad \text{ if } \qquad \sum_{j\in \N}\lambda_j^{-\frac s2+l_0 + l_5-1}<\infty.$$ Now using $\lambda_j\sim c(d) j^{4/d}$ (see Proposition~3.8 of \cite{mm_scaling}) we obtain that $\sum_{j\in \N}\lambda_j^{-\frac s2+l_0 + l_5-1}$ is finite whenever $s>s_d$. Thus we have proved \eqref{mm_limsup_psi}. 

{\bf Case 2: $\kappa\sim 2d N^2$.} Due to the compact embedding of the spaces $\mathcal H^{-s}_{-\Delta+\Delta^2}(D)$, to show that the sequence $(\Psi_N)_{N\in \mathbb{N}}$ is tight in $\mathcal H^{-s}_{-\Delta+\Delta^2}(D)$ for all $s>s_d$, it is enough to show that
\begin{equation}\label{mixed_limsup_psi}
\limsup_{N\to \infty}\E_{\Lambda_N}[\|\Psi_N\|_{-s,\,-\Delta + \Delta^2}^2]<\infty \quad \forall \, s>s_d.
\end{equation}
As in the previous case, by definition of dual norm we have 
$$\E_{\Lambda_N}\left[\|\psi_N\|_{-s, \,-\Delta + \Delta^2}^2\right] \le \sum_{j\in \N}\mu_j^{-s/2}\E_{\Lambda_N}[(\psi_N\,,\,v_j)^2].$$
Note that $u=\mu_j^{-1}v_j$ is the unique solution of \eqref{eqa:continuum} with $L=-\Delta_c+\Delta_c^2$ for $f:=u_j$. Define $e_{N,j}$ to be the error between the solution of the discrete Dirichlet problem~\eqref{mixed_H_N} and the continuum one~\eqref{eqa:continuum} with $f:=v_j$. Now as in~\eqref{eq:varbound} we have
$$\E_{\Lambda_N}[(\Psi_N\,,\,v_j)^2]\le C\sup_{x\in D}|v_j(x)| \left({N^{-d} \sum_{x\in \frac1N\Lambda_N} e_{N,j}(x)^2}\right)^{1/2}+ C\mu_j^{-1}\left(\sup_{x\in D}|v_j(x)|\right)^2.$$
Using Theorem~\ref{approx_result}~\ref{thm:mixed_one} along with the bounds~\eqref{bound0:mixed} we obtain
	\begin{align*}
	\E_{\Lambda_N}[(\Psi_N\,,\,v_j)^2]&\le C \mu_j^{l_0} \left[\mu_j^{2l_5-2}N^{-2} + \mu_j^{2l_5-2}\left(\frac{\kappa}{2dN^2}-1\right)^2 +\mu_j^{2l_2-2} N^{-1}\right]^{\frac12} + C\mu_j^{2l_0-1}\\
	&\le C \mu_j^{l_0 + l_5-1}.
	\end{align*}
	Therefore we have 
	\begin{align*}
	\E_{\Lambda_N}\left[\|\Psi_N\|_{-s, \,-\Delta + \Delta^2}^2\right]\le C \sum_{j\in \N}\mu_j^{-\frac s2}\mu_j^{l_0 + l_5-1}.
	\end{align*}
	Thus 
	$$\limsup_{N\to \infty}\E_{\Lambda_N}[\|\Psi_N\|_{-s, \,-\Delta + \Delta^2}^2]<\infty \qquad \text{ if } \qquad \sum_{j\in \N}\mu_j^{-\frac s2+l_0 + l_5-1}<\infty.$$ 
	From Proposition \ref{prop:mixed_Weyl} we obtain that $\sum_{j\in \N}\mu_j^{-\frac s2+l_0 + l_5-1}<\infty$ whenever $s>s_d$. Thus we have proved \eqref{mixed_limsup_psi}.	
	
{\bf Case 3: $\kappa \ll N^{2}$.} The arguments are similar to the previous two cases and hence we just indicate the required bounds. To show tightness in $\mathcal H^{-s}_{-\Delta}(D)$ it is enough to show
	\begin{equation}\label{gff_limsup_psi}
	\limsup_{N\to \infty}\E_{\Lambda_N}[\|\Psi_N\|_{-s,\,-\Delta}^2]\le \sum_{j\in \N}\nu_j^{-s}\E_{\Lambda_N}[(\Psi_N\,,\,w_j)^2] <\infty \quad \forall \, s>{d}/2 + \floor{{d}/2}+2.
	\end{equation}
Setting $e_{N,j}$ to be the error between the solution of the discrete Dirichlet problem~\eqref{gff_H_N} and the continuum one \eqref{eqa:continuum} with $f:=w_j$ we obtain
$$
	\E_{\Lambda_N}[(\Psi_N\,,\,w_j)^2]\le C\sup_{x\in D}|w_j(x)| \left({N^{-d} \sum_{x\in \frac1N\Lambda_N} e_{N,j}(x)^2}\right)^{1/2}+ C\nu_j^{-1}\left(\sup_{x\in D}|w_j(x)|\right)^2.
$$
	Using Theorem~\ref{approx_result}~\ref{thm:gff_one} along with the bounds~\eqref{bound0:GFF} we can conclude the following upper bound for $\E_{\Lambda_N}[(\Psi_N\,,\,w_j)^2]$:
	\begin{align*}
	&C \sup_{x\in D}|w_j(x)| \left[ \left(\nu_j^{-1}\nu_j^{\frac{\floor{\frac{d}2}+5}2}\right)^2 \delta^4 + \left(\nu_j^{-1}\nu_j^{\frac{\floor{\frac{d}2}+3}2}
	\right)^2 \frac{\delta\kappa}{2dN^2} + \left(\nu_j^{-1}\nu_j^{\frac{\floor{\frac{d}2}+2}2}\right)^2 \delta\right]^{\frac12}+C\nu_j^{-1}\left(\sup_{x\in D}|w_j(x)|\right)^2 ,
	\end{align*}
	where $\delta = \max\{\frac{1}{N},\frac{\sqrt{\kappa}}{\sqrt{2d}N}\}$. Now a consequence of the above and \eqref{bound0:GFF} is that
	\begin{align}\label{eq:second_equation}
	\E_{\Lambda_N}[(\Psi_N\,,\,w_j)^2] \le C\nu_j^{\floor{\frac{d}2}+2}.
	\end{align}
	Therefore we have 
	\begin{align*}
	\E_{\Lambda_N}\left[\|\Psi_N\|_{-s,\,-\Delta}^2\right]\le C \sum_{j\in \N}\nu_j^{-s}\nu_j^{\floor{\frac{d}2}+2}.
	\end{align*}
	Thus 
	$$\limsup_{N\to \infty}\E_{\Lambda_N}[\|\Psi_N\|_{-s,\,-\Delta}^2]<\infty \qquad \text{ if } \qquad\sum_{j\in \N}\nu_j^{-s+\floor{\frac{d}2}+2}<\infty.$$ But $\nu_j\sim Cj^{\frac2d}$ and $\sum_{j\in \N}j^{\frac2d(-s+\floor{\frac{d}2}+2)}<\infty$ whenever $s>{d}/2 + \floor{{d}/2}+2$. Thus we have proved \eqref{gff_limsup_psi}.
	
For all the cases we now have the tightness and the convergence of $(\Psi_N, f)$ for all $f\in C_c^\infty(D)$. A standard uniqueness argument completes the proof of Theorem~\ref{thm:main}, using the fact that $C_c^\infty(D)$ is dense in $\mathcal H_{\Delta^2,0}^s(D)$, $\mathcal H_{-\Delta+\Delta^2, 0}^s(D)$ and $\mathcal H_{-\Delta, 0}^s(D)$ respectively.


\section{\bf Proof of Theorem~\ref{thm:main2}}\label{sec:main2_proofs}
In this section we prove Theorem~\ref{thm:main2} by showing finite dimensional convergence and tightness. The proof is similar to the proofs of the lower dimensional results in \cite{mm_scaling} and \cite{mixed_scaling} and hence we shall only state the important bounds needed for the proof. 
One can show tightness of the sequence using Theorem 14.9 of \cite{kallenberg:foundations} (see also Theorem 2.5 of \cite{mm_scaling}). To use this result one mainly needs bounds on the increments of the following type: $$\E_{\Lambda_N}\left[ \left| \Psi_N(t)- \Psi_N(s)\right|^2\right]\le C \|t-s\|^{1+b}, \quad t,s \in \overline{D},\,b\ge 0.$$
Such bounds can be obtained using the Brascamp-Lieb inequality and the following Lemma which is proved by using the estimates for the membrane model and the discrete Gaussian free field.
\begin{lemma}~\label{DGF_bound0}
Let $\prob_{\Lambda_N}^{MM}$ and $\prob_{\Lambda_N}^{GFF}$ denote respectively, the law of the membrane model and the discrete Gaussian free field on $\Lambda_N$ with zero boundary conditions outside $\Lambda_N$.
\begin{enumerate}[leftmargin=*,label=(\Roman*),ref=(\Roman*)]
\item\label{DGF_bound}Let $\kappa\gg N^2$ or $\kappa \sim 2dN^2$ and $1\le d\le 3$. Then for all $x\in\Z^d$
\begin{enumerate}[label=(\arabic*),ref=(\arabic*)]
	\item\label{1.DGF_bound} 
		$$G_{\Lambda_N}(x,x)\le \kappa^{-1}\E_{\Lambda_N}^{MM}(\vr_x^2) \le C\kappa^{-1}N^{4-d}.$$
	\item\label{2.DGF_bound0}  
$$\E_{\Lambda_N}\left[\left( \vr_{x+e_i}-\vr_x\right)^2 \right] \le \kappa^{-1}\E^{MM}_{\Lambda_N}\left[\left(\vr_{x+e_i} - \vr_x\right)^2\right]\le \left\lbrace 
\begin{array}{l l l}
C\kappa^{-1}N &  d=1\\
C\kappa^{-1}\log N &  d=2\\
C \kappa^{-1}&  d=3
\end{array} .\right.$$
\end{enumerate}

\item Let $\kappa\ll N^{2}$ and $d=1$. Then for all $x\in\Z^d$
\begin{enumerate}[label=(\arabic*),ref=(\arabic*)]
	\item 
	$$G_{\Lambda_N}(x,\,x)\le \E_{\Lambda_N}^{GFF}(\vr_x^2) \le C N. $$
	\item 
	\begin{align*}
	\E_{\Lambda_N}[(\vr_{x +e_i} - \vr_x)^2] \le \E_{\Lambda_N}^{GFF}[(\vr_{x+e_i} - \vr_x)^2] \le C.
	\end{align*}
\end{enumerate}
\end{enumerate}
\end{lemma}
\begin{proof}~
\begin{enumerate}[leftmargin=*,label=(\Roman*),ref=(\Roman*)]
\item To show the first bound one can first show using Theorem 5.1 of \cite{brascamp_lieb} that 
\begin{align*}
G_{\Lambda_N}(x,\,x) \le \kappa^{-1}\E_{\Lambda_N}^{MM}(\vr_x^2).
\end{align*}
The bound for the $d=1$ case can be obtained using the random walk representation of the model used in Lemma~\ref{lem:one_mm}. For $d=2,\,3$ we obtain the bound from Theorem 1.1 of \cite{Mueller:Sch:2017}.

For the second part the Brascamp-Lieb inequality yields 
\begin{align*}
\E_{\Lambda_N}[(\vr_{x+e_i} - \vr_x)^2]\le \kappa^{-1}\E^{MM}_{\Lambda_N}[(\vr_{x+e_i} - \vr_x)^2].
\end{align*}
The bound now follows from Lemma~\ref{lem:one_mm} (for $d=1$) and Theorem 1.1 of \cite{Mueller:Sch:2017} (for $d=2,3$).

\item The argument in this case are similar to the above case. The bounds in this case are obtained using the Brascamp-Lieb inequality and an argument similar to the proof of Lemma 12 of \cite{mixed_scaling}.\qedhere
\end{enumerate}
\end{proof}
The detailed argument for tightness is similar to that in \citet[Section 2]{mm_scaling} and hence skipped.

To conclude the finite dimensional convergence we first show the convergence of the covariance. We shall discuss the argument for the cases when $\kappa\gg N^2$  and $\kappa\sim 2d N^2$. The argument for both cases is the same.  In the other instance, that is $\kappa\ll N^{2}$, we can argue similarly using the following additional piece of information: the covariance function $$G_D(x,y)=\min\{x,y\} -xy,\quad x,y\in\overline D$$ of the Brownian bridge is nothing but the Green's function for the problem 
\begin{align*}
\begin{cases}
-\frac{\De^2 u}{\De x^2}(x) = f(x) & x\in D\\
u(x)=0 & x\in\partial D.
\end{cases}
\end{align*}

Suppose now $\kappa\gg N^2$  or $\kappa\sim 2d N^2$. For $x,y\in \overline D\cap N^{-1}\Z^d$ we define
 $$G_{\frac1N}(x,y):=(2d)^{-2}\kappa N^{d-4}G_{\Lambda_N}(Nx,Ny).$$ 
 
 We now interpolate $G_{\frac1N}$ in a piece-wise constant fashion on small squares of $\overline D \times \overline D$ to get a new function $G_{\frac1N}^I$.
 We show that $G_{\frac1N}^I$ converges uniformly to $G_D$ on $\overline D\times\overline D$. Indeed, let $F_N:=G_{\frac1N}^I - G_D$. Similarly as in the proof of the finite dimensional convergence in Theorem~\ref{thm:main}~\ref{thm:MM} or Theorem~\ref{thm:main}~\ref{thm:MIXED} it follows that, for any $f,\,g\in C_c^\infty(D)$,
\begin{align*}
\lim_{N\to\infty}\sum_{x, y\in \frac1N D_N} N^{-2d} G^I_{\frac1N}(x,y) f(x) g(y)= \iint_{D\times D} G_D(x,y)f(x)g(y)\De x \De y.
\end{align*}
Again from Riemann sum convergence we have 
\begin{align*}
\lim_{N\to\infty}\sum_{x, y\in \frac1ND_N} N^{-2d} G_D(x,y) f(x) g(y)=\iint_{D\times D} G_D(x,y)f(x)g(y)\De x \De y.
\end{align*}
Thus we get
\begin{align}
\lim_{N\to\infty}\sum_{x, y\in \frac1N D_N} N^{-2d} F_N(x,y) f(x) g(y)=0\label{eq:unique_zero_m}.
\end{align}
Note that $G_D$ is bounded and $$\sup_{x,y\in\frac1N D_N}|G_{\Lambda_N}(Nx,Ny)| \le C\kappa^{-1}N^{4-d}.$$
These imply that $$\sup_{x,y\in\overline D} |F_N(x,y)|\le C.$$
Thus $F_N$ has a subsequence converging uniformly to some function $F$ which is bounded by $C$. With abuse of notation we denote this subsequence by $F_N$. We then have 
\begin{align*}
\lim_{N\to\infty}\sum_{x, y\in \frac1N D_N} N^{-2d} F_N(x,y) f(x) g(y)= \iint_{D\times D}  F(x,y)f(x)g(y)\De x \De y.
\end{align*}
Uniqueness of the limit gives
$$\iint_{D\times D}  F(x,y)f(x)g(y)\De x \De y = 0$$
 by~\eqref{eq:unique_zero_m}. From this we obtain that $F(x,y)=0$ for almost every $x$ and almost every $y$. The definition by interpolation of $G^I_{\frac1N}$ ensures that $F$ is pointwise equal to zero. Finally, the fact that the original sequence $F_N$ converges uniformly to zero follows using the subsequence argument. 

We now show the finite dimensional convergence. First let $t\in\overline D$. We write 
$$\Psi_N(t)=\Psi_{N,1}(t) + \Psi_{N,2}(t)$$
where $\Psi_{N,1}(t):=(2d)^{-1}\sqrt{\kappa}N^{\frac{d-4}2}\vr_{\floor{Nt}}$ and $\Psi_{N,2}(t):=\psi_N(t)-\psi_{N,1}(t)$. From Lemma~\ref{DGF_bound0}\ref{DGF_bound}\ref{2.DGF_bound0} it follows that $\E_{\Lambda_N}[\Psi_{N,2}(t)^2]$ goes to zero as $N$ tends to infinity. Therefore to show that $\Psi_N(t)$ converges in distribution it is enough to show that $\var[\Psi_{N,1}(t)]\to G_D(t,t)$. But we have
\begin{align*}
\var[\Psi_{N,1}(t)]=(2d)^{-2} \kappa N^{d-4}G_{\Lambda_N} \left(\floor{Nt},\,\floor{Nt}\right)=  G^I_{\frac1N}(t,t)\to G_D(t,t)
\end{align*}
since the sequence $F_N$ converges to zero uniformly.
Since the variables under consideration are Gaussian, one can show the finite dimensional convergence using the convergence of the Green's functions. This completes the proof of Theorem~\ref{thm:main2}. \qed

\section{\bf Proof of Theorem \ref{approx_result}}\label{proof_approx_result}
This section is devoted to proof of the error estimation result in Theorem~\ref{approx_result}. To estimate the error we need to develop some Sobolev inequalities in the general setting which involve consistency between discrete and continuous operators. The content of this section can be of independent interest and can possibly be applied to general interface models. We would like to stress that although we follow the ideas involved in~\cite{thomee}, we cannot quote the results from there {\em verbatim} as the coefficients of the discrete operators do not depend on the scaling of the lattice. 
Also another important remark is that the discrete Dirichlet problem involving the operators $L_h$ introduced in \eqref{discrete_op} requires two boundary conditions, but the definition of the limiting operator $-\Delta_c$ involves only one boundary condition. The ideas from \cite{thomee} work well when $L=\Delta_c^2$ or $L=-\Delta_c+\Delta_c^2$. In the case when $L=-\Delta_c$, we assign a cut-off which helps in controlling the error around the boundary. The proof of Theorem~\ref{approx_result}~\ref{thm:gff_one} should be applicable to many other models. 

\subsection{Sobolev-type norm inequalities} The main aim of this Subsection is to have an estimate on the $\ell^2$ norm of a function on the grid in terms of the operator $L_h$ (and its truncated version). Later this turns out to be useful as we use the convergence of $L_h$ to $L$. We continue with all the definitions and notations from Section~\ref{main_ingredient}.

The notion of discrete forward and backward derivatives will be essential in the following arguments.
\begin{align*}
& \partial_ju(x) := \frac1h(u(x+he_j)-u(x)),\\
& \bar\partial_ju(x) := \frac1h(u(x)-u(x-he_j)),\\
& \partial^\alpha:=\partial_1^{\alpha_1}\cdots\partial_d^{\alpha_d},\\
& \bar\partial^\alpha:=\bar\partial_1^{\alpha_1}\cdots\bar\partial_d^{\alpha_d},
\end{align*}
where $\alpha=(\alpha_1,\ldots,\alpha_d)$ is a multi-index.
It is easy to see that
\begin{align*}
 \la \partial_j u\,,\,v \ra_{h,\,grid} = \la u\,,\,\bar\partial_jv \ra_{h,\,grid}
\end{align*}
for grid-functions vanishing outside a finite set.
We now define
$$\|u\|_{h,m}:=\left( \sum_{|\alpha|\le m} \|\partial^\alpha u \|_{h,\,grid}^2\right)^{\frac12}$$
and obtain the following Lemma.

\begin{lemma}[{\citet[Lemma~3.1]{thomee}}]\label{lem:3.1}
	There are constants $C=C_j$ independent of $u$ and $h$ such that
	\begin{equation}\label{eq:3.1.1}
	\|u\|_{h,\,grid}\le C\|\partial_j u\|_{h,\,grid}, \quad u\in\mathcal D_h,\,\,j=1,\ldots, d,
	\end{equation}
	and for fixed $m\ge 1$,
	\begin{equation}\label{eq:3.1.2}
	\|u\|_{h,\,grid}\le C\| u\|_{h,m}, \quad u\in\mathcal D_h.
	\end{equation}
\end{lemma}

We will need the following norm which rescales the function near the boundary:
$$|||u|||_{h,m}:= \left( h^d\left(\sum_{x\in R_h^*} u(x)^2 + \sum_{x\in B_h^*} (h^{-m} u(x))^2\right)\right)^{\frac12}, \quad u\in\mathcal D_h.$$
We can relate the weighted Sobolev norm $|||\cdot|||_{h,m}$ to $||\cdot||_{h,m}$ with this bound:
\begin{lemma}[{\citet[Lemma~3.4]{thomee}}]\label{lem:3.4}
	There is a constant $C$ independent of $u$ and $h$ such that
	\begin{align*}
	|||u|||_{h,m} \le C \|u\|_{h,m}, \quad u\in\mathcal D_h.
	\end{align*}
\end{lemma}

We rewrite $L_h$ in \eqref{discrete_op} as
\begin{equation}\label{eq:newL_h}L_h u(x) = h^{-2m}\sum_{\eta} c_\eta u(x+\eta h),\end{equation}
where $\eta=(\eta_1,\ldots,\eta_d)$ with the $\eta_j$'s being integers and the $c_\eta$'s being real numbers which may depend on $h$. We now define the characteristic polynomial of $L_h$ by
\begin{align}\label{char_poly}
p(\theta):= \sum_{\eta} c_\eta \e^{\iota \la \eta\,,\, \theta \ra},
\end{align}
where $\theta=(\theta_1,\ldots,\theta_d)$ and $\la \eta\,,\, \theta \ra = \sum_{j=1}^d \eta_j \theta_j$.
We have the following Lemma:
\begin{lemma}\label{lem3.2}
\begin{align*}
\la L_hu\,,\,u\ra_{h,\,grid} = h^{d-2m}(2\pi)^{-d}\int_S p(\theta) |\hat{u}(\theta)|^2\mathrm{d}\theta,\quad u\in\mathcal D_h.
\end{align*}
where $$\hat{u}(\theta) = \sum_{\xi\in\Z^d} u(\xi h)\e^{-\iota \la \xi\,,\, \theta \ra}$$ and $S= \{\theta: |\theta_j|\le \pi,\, j=1,\ldots,d\}.$
\end{lemma}
\begin{proof}We expand
\begin{align*}
\la L_hu\,,\,u\ra_{h,\,grid} &= h^d \sum_{x\in h\Z^d} L_hu(x)u(x)\\
& \stackrel{\eqref{eq:newL_h}}{=} h^{d-2m}\sum_{x\in h\Z^d}\sum_{\eta\in\Z^d} c_\eta u(x+\eta h) u(x)\\
& = h^{d-2m}\sum_{x,\,\xi\in h\Z^d} c_{\frac{\xi-x}h} u(\xi) u(x).
\end{align*}
By inverting \eqref{char_poly} we have $$c_\eta = (2\pi)^{-d} \int_S p(\theta) \e^{-\iota \la \eta,\,\theta \ra} \mathrm{d}\theta.$$
Thus 
\begin{align*}
\la L_hu\,,\,u\ra_{h,\,grid} &= h^{d-2m}\sum_{x,\,\xi\in h\Z^d} (2\pi)^{-d} \int_S p(\theta) \e^{-\iota \la \frac{\xi-x}h,\,\theta \ra} \mathrm{d}\theta u(\xi) u(x)\\
& =  h^{d-2m}(2\pi)^{-d}\int_S p(\theta) |\hat{u}(\theta)|^2\mathrm{d}\theta.\qedhere
\end{align*}
\end{proof}
We will also need
\begin{lemma}[{\citet[Lemma~3.3]{thomee}}]\label{lem:3.3}
There is a constant $C$ independent of $u$ and $h$ such that
\begin{align*}
\|u\|_{h,m}^2 \le C \sum_{j=1}^d \|\partial_j^m u\|_{h,\,grid}^2,\quad u\in\mathcal D_h.
\end{align*}
\end{lemma}
\begin{proof}
We first prove that if $\alpha$ is a multi-index with $|\alpha|=m$ then
\begin{equation}\label{lab1}
\la \bar\partial^\alpha\partial^\alpha u,\,u \ra_{h,\,grid} \le \la Q_h u,\,u \ra_{h,\,grid},\quad u\in\mathcal D_h,
\end{equation}
where $Q_h$ is the difference operator
\begin{equation}Q_h u := \sum_{j=1}^d \bar\partial_j^m \partial_j^m u. \label{Q_h}\end{equation}
Similar to \eqref{char_poly} we can show the characteristic polynomial of $\bar\partial^\alpha\partial^\alpha$ and $Q_h$ are respectively $$q_1(\theta) = 2^m \prod_{j=1}^d (1-\cos \theta_j)^{\alpha_j}$$
and
$$q_2(\theta) = 2^m \sum_{j=1}^d (1-\cos \theta_j)^m.$$
Now by the inequality between arithmetic and geometric mean we have
$$q_1(\theta) \le 2^m \sum_{j=1}^d m^{-1}\alpha_j(1-\cos \theta_j)^m \le q_2(\theta).$$
Using Lemma \ref{lem3.2} we obtain \eqref{lab1}, which implies
$$\|\partial^\alpha u\|_{h,\,grid}^2 \le \sum_{j=1}^d \|\partial_j^m u\|_{h,\,grid}^2,\quad u\in\mathcal D_h.$$
For $|\alpha|<m$, one can show using Lemma \ref{lem:3.1}
$$\|\partial^\alpha u\|_{h,\,grid}^2 \le C \sum_{j=1}^d \|\partial_j^m u\|_{h,\,grid}^2,\quad u\in\mathcal D_h.$$
Hence the proof is complete.
\end{proof}
\subsection{Errors in the Dirichlet problem} We have shown some discrete Sobolev inequalities till now. We now relate these directly to our discrete operators. We start dealing with each of the operators separately. Before we do so let us show here the existence and uniqueness of the solution of the discrete boundary value problem \eqref{eq:discrete}-\eqref{eq:discrete boundary}.
\begin{lemma}\label{fact:unique_discrete}
	The finite difference Dirichlet problem \eqref{eq:discrete}-\eqref{eq:discrete boundary} has exactly one solution for arbitrary $f$.
\end{lemma}
\begin{proof}
We first show the following. There exists a constant $C>0$ independent of $u$ and $h$ such that 
\begin{equation}\label{eq:revised_00}
 	     \|u\|_{h,\,grid}\le C\|L_{h}u\|_{h,\,grid},\quad u\in\mathcal D_h.
\end{equation}
In case $L= \Delta_c^2$ or $-\Delta_c + \Delta_c^2$,~\eqref{eq:revised_00} follows Lemma~\ref{lem:3.1} and from the proof of Lemmas~\ref{thm:mm_4.2}, \ref{thm:mixed_4.2} respectively. For $L=-\Delta_c$ the argument is similar once we observe that
\begin{align*}
	p(\theta)&=-\sum_{i=1}^d (2\cos{\theta_i}-2)+\frac{\rho_1(h)}{h^2}\sum_{i,\,j=1}^d[ 2\cos{(\theta_i + \theta_j)} + 2\cos{(\theta_i -\theta_j)}-4\cos{\theta_i}-4\cos{\theta_j}+4]\\
&=\sum_{i=1}^d (2-2\cos{\theta_i})+\frac{\rho_1(h)}{h^2}\sum_{i,\,j=1}^d[4(1-\cos{\theta_i})(1-\cos{\theta_j})]\\
&\ge 2\sum_{i=1}^d (1-\cos{\theta_i}).
\end{align*}
Now since $u\equiv 0$ in $B_h$, Equation~\eqref{eq:discrete} can be considered as a linear system of equations with the same number of equations as of unknowns (the number of points in $R_h$). Therefore it is sufficient to prove that the corresponding homogeneous system has only the trivial solution i.e. $u\equiv 0$ in $R_h$. This follows from~\eqref{eq:revised_00}.
\end{proof}

\subsubsection{Bilaplacian case: proof of Theorem \ref{approx_result}~\ref{thm:mm_one}}
In this subsection we consider $L:=\Delta_c^2$. Recall $\rho_2(h)\to 0$ and we  have for $x\in h\mathbb{Z}^d$,
    \begin{align*}
     L_hu(x)&=\frac1{h^4}\left[-h^{2}\rho_2(h)\sum_{i=1}^d (u(x+he_i)+u(x-he_i)-2u(x))\right.\\
     &\left.+\sum_{i,\,j=1}^d\left\lbrace u(x+h(e_i+e_j))+u(x-h(e_i+e_j))+u(x+h(e_i-e_j))+u(x-h(e_i-e_j))\right.\right.\\
    &\left.\left.-2(u(x+he_i)-2u(x-he_i)-2(u(x+he_j)-2u(x-he_j) +4u(x))\right\rbrace\right] .
    \end{align*}
We define the operator $L_{h,2}$ as follows:
\begin{equation}\label{eq:Lh2}L_{h,\,2} f(x) = \begin{cases}
 	L_h f(x)  & x\in R_h^\ast\\
 		h^2 L_h f(x)  & x\in B_h^\ast\\
		0 & x\notin R_h.
 		\end{cases}\end{equation}
Then we have the following Lemma involving $L_{h,\,2}$.
\begin{lemma}\label{thm:mm_4.2}
 	  There exists a constant $C>0$ independent of $u$ and $h$ such that 
	    \[
 	     \|u\|_{h,\,2}\le C\|L_{h,\,2}u\|_{h,\,grid},\quad u\in\mathcal D_h.
	    \]
\end{lemma}
\begin{proof}
    We consider the characteristic polynomial of $L_h$ and observe that
    \begin{align*}
p(\theta)&=-h^{2}\rho_2(h)\sum_{i=1}^d (2\cos{\theta_i}-2)\\
&+\sum_{i,\,j=1}^d[ 2\cos{(\theta_i + \theta_j)} + 2\cos{(\theta_i -\theta_j)}-4\cos{\theta_i}-4\cos{\theta_j}+4]\\
&=h^{2}\rho_2(h)\sum_{i=1}^d (2-2\cos{\theta_i})+\sum_{i,\,j=1}^d[4(1-\cos{\theta_i})(1-\cos{\theta_j})]\\
&\ge 4 \sum_{i=1}^d (1-\cos{\theta_i})^2.
\end{align*}
Hence by Lemmas~\ref{lem:3.3} and~\ref{lem3.2} we obtain for $u\in\mathcal D_h$
\begin{align*}
\|u\|_{h,2}^2 \le C \sum_{j=1}^d \|\partial_j^2 u\|_{h,\,grid}^2 = C \la Q_h u,\,u \ra_{h,\,grid} \le C \la L_hu,\,u\ra_{h,\,grid},
\end{align*}
where $Q_h$ is the difference operator defined in \eqref{Q_h} with $m=2$.
	Again we have 
	$$\la L_hu,\, u\ra_{h,\,grid} = h^d \left[ \sum_{x\in B_h^*} L_{h,\,2} u(x)\left(h^{-2}u(x)\right) + \sum_{x\in R_h^*} L_{h,\,2} u(x) u(x) \right]$$
	Therefore by Cauchy-Schwarz inequality we have
	$$|\la L_hu,\, u\ra_{h,\,grid}| \le C \|L_{h,\,2}u\|_{h,\,grid} |||u|||_{h,\,2}.$$
	Thus from Lemma \ref{lem:3.4} we have
	\begin{align*}
	\|u\|^2_{h,\,2} \le C \|L_{h,\,2}u\|_{h,\,grid} \, |||u|||_{h,\,2} \le C \|L_{h,\,2}u\|_{h,\,grid} \, \|u\|_{h,\,2}
	\end{align*}
	This completes the proof.
\end{proof}

%

We have now all the ingredients to show Theorem~\ref{approx_result}~\ref{thm:mm_one}.
 
 \begin{proof}[Proof of Theorem~\ref{approx_result}~\ref{thm:mm_one}]
We denote all constants by $C$ and they do not depend on 
 		$u,\, f$.
 		Using Taylor expansion we have for all $x\in R_h$ and for small $h$
 		$$L_hu(x)=h^{-2}\rho_2(h)\mathcal R_2(x) + Lu(x)+h^{-4}\mathcal R_5(x)$$
 		where $\lvert \mathcal R_2(x)\rvert\leq CM_2h^2$ and $\lvert \mathcal R_5(x)\rvert\leq CM_5h^5$. We thus obtain, for $x\in R_h$,
 		\begin{align}
 		L_he_h(x)&=L_hu(x)-L_hu_h(x)\nonumber\\
 		&= h^{-2}\rho_2(h)\mathcal R_2(x) + h^{-4}\mathcal R_5(x).\label{eq:sokhi_bhabona}
 		\end{align}
 		 		For $x\in R^*_h$ we have
 		\begin{align*}
 		L_{h,2}R_he_h(x)&= L_hR_he_h(x)=L_he_h(x)=h^{-2}\rho_2(h)\mathcal R_2(x) + h^{-4}\mathcal R_5(x).
 		\end{align*}
 		For $x\in B^*_h$ at least one among $x\pm h(e_i\pm e_j),\,x\pm he_i$ is in $B_h\setminus \partial D$. For any $y\in B_h\setminus\partial D$ we consider a point $b(y)$ on $\partial D$ of minimal distance to $y$. Note that this distance is at most $2h$. Now using Taylor expansion and the fact that the value of $u$ and all its first order derivatives are zero at $b(y)$ one sees that $$u(y)=u_h(y)+\mathcal R^{'}_2(y)$$ where $\lvert \mathcal R^{'}_2(y)\rvert\leq CM_2h^2$. For $x\in B_h^\ast$ denote by $S(x)$ the neighbors of $x$ which are in $B_h\setminus \partial D$ i.e. $$S(x)= \{ y: y\in B_h\setminus \partial D \cap \{ x\pm h e_i, x\pm h(e_i\pm e_j): 1\le i,j \le d\}\}.$$ Therefore, for $x\in B_h^\ast$,
 		\begin{align}
 		L_{h,2} R_he_h(x)&= h^2 L_h R_he_h(x)\nonumber\\
 		&=h^2 \left\{L_he_h(x)- h^{-4} \sum_{y\in S(x)}\left(h^2\rho_2(h)C(y)e_h(y) + C^{'}(y) e_h(y)\right)\right\}\nonumber\\
 		&\stackrel{\eqref{eq:sokhi_bhabona}}{=} h^2\{h^{-2}\rho_2(h)\mathcal R_2(x) + h^{-4}\mathcal R_5(x)\} + (C \rho_2(h) + C^{'}h^{-2}) \mathcal R^{''}_2(x)\label{eq:long_one}
 		\end{align}
 	    where $\lvert \mathcal R^{''}_2(x)\rvert\leq CM_2h^2$. Hence
 	    \begin{align*}
 	    \|L_{h,2} R_he_h\|_{h,\,grid}^2 & 
 	    \stackrel{\eqref{eq:long_one}}{=} h^d \left[ \sum_{x\in R_h^\ast} \left(h^{-2}\rho_2(h)\mathcal R_2(x) + h^{-4}\mathcal R_5(x)\right)^2 \right.\\
 	    &\left. +\sum_{x\in B_h^\ast} \left(\rho_2(h)\mathcal R_2(x) + h^{-2}\mathcal R_5(x) + (C \rho_2(h) + C^{'}h^{-2}) \mathcal R^{''}_2(x) \right)^2\right]\\
 	    &\le C h^d \left[ \sum_{x\in R_h^\ast} \left(M_2^2(\rho_2(h))^2 + M_5^2h^2\right) +\sum_{x\in B_h^\ast} \left(M_2^2h^{4}(\rho_2(h))^2 + M_5^2h^6+ M_2^2\right)\right]\\
 	    &\le C\left[\left(M_2^2(\rho_2(h))^2 + M_5^2h^2\right) + h \left(M_2^2h^{4} (\rho_2(h))^2 + M_5^2h^6+ M_2^2\right)\right]
 	    \end{align*}
 	    where the last inequality holds as the number of points in $B_h^\ast$ is $O(h^{-(d-1)})$.
 	    Finally to complete our proof we obtain
 	    \begin{align*}\label{eq4N:errorbound}
 	    \|R_he_h\|_{h,\,grid}^2 &\le C\left[M_2^2(\rho_2(h))^2 + M_5^2h^2  + M_2^2h^{5} (\rho_2(h))^2 + M_5^2h^7+ M_2^2 h\right]\\
 	    &\le C\left[M_5^2h^2 + M_2^2(\rho_2(h))^2 +  M_2^2 h\right]
 	    \end{align*}
 	    using Lemmas \ref{lem:3.1} and \ref{thm:mm_4.2}.  
 \end{proof}

\subsubsection{Laplacian + Bilaplacian case: proof of Theorem~\ref{approx_result}~\ref{thm:mixed_one}} 	
In this subsection we consider $L=-\Delta_c + \Delta_c^2$. Recall $\rho_3(h) \to 1$ and we have for $x\in h\mathbb{Z}^d$,
    \begin{align*}
     L_hu(x)&=\frac1{h^4}\left[-h^{2}\sum_{i=1}^d (u(x+he_i)+u(x-he_i)-2u(x))\right.\\
     &\left.+ \rho_3(h)\sum_{i,\,j=1}^d\left\lbrace u(x+h(e_i+e_j))+u(x-h(e_i+e_j))+u(x+h(e_i-e_j))+u(x-h(e_i-e_j))\right.\right.\\
     &\left.\left.-2(u(x+he_i)-2u(x-he_i)-2(u(x+he_j)-2u(x-he_j) +4u(x))\right\rbrace\right].
    \end{align*}
 We define the operator $L_{h,2}$ as in~\eqref{eq:Lh2} and obtain
\begin{lemma}\label{thm:mixed_4.2}
 	  There exists a constant $C>0$ independent of $u$ and $h$ such that 
	    \[
 	     \|u\|_{h,\,2}\le C \|L_{h,\,2}u\|_{h,\,grid},\quad u\in\mathcal D_h.
	    \]
\end{lemma}
\begin{proof}
    We observe that
    \begin{align*}
p(\theta)&=-h^{2}\sum_{i=1}^d (2\cos{\theta_i}-2)+ \rho_3(h)\sum_{i,\,j=1}^d[ 2\cos{(\theta_i + \theta_j)} + 2\cos{(\theta_i -\theta_j)}-4\cos{\theta_i}-4\cos{\theta_j}+4]\\
&=h^{2}\sum_{i=1}^d (2-2\cos{\theta_i})+ \rho_3(h)\sum_{i,\,j=1}^d[4(1-\cos{\theta_i})(1-\cos{\theta_j})]\\
&\ge 4\rho_3(h)\sum_{i=1}^d (1-\cos{\theta_i})^2.
\end{align*}
Hence by Lemma \ref{lem:3.3} and \ref{lem3.2} we obtain for $u\in\mathcal D_h$
\begin{align*}
\|u\|_{h,2}^2 \le C \sum_{j=1}^d \|\partial_j^2 u\|_{h,\,grid}^2 = C \la Q_h u,\,u \ra_{h,\,grid} \le C (\rho_3(h))^{-1} \la L_hu\,,\,u\ra_{h,\,grid} \le C\la L_hu\,,\,u\ra_{h,\,grid},
\end{align*}
where $Q_h$ is the difference operator defined in \eqref{Q_h} with $m=2$.
The rest of the proof is similar to Lemma \ref{thm:mm_4.2} and hence omitted.
\end{proof}

%

We now prove the approximation result in this case.
\begin{proof}[Proof of Theorem~\ref{approx_result}~\ref{thm:mixed_one}]
As before the constant $C$ does not depend on 
$u$ and $f$. Using Taylor expansion we have for all $x\in R_h$ and for small $h$
$$L_hu(x)=Lu(x)+ (\rho_3(h)-1)\Delta_c^2 u(x) + h^{-2} \mathcal R_4(x)+\rho_3(h)h^{-4}\mathcal R_5(x)$$
where $\lvert \mathcal R_4(x)\rvert\leq CM_4h^4,\,\lvert \mathcal R_5(x)\rvert\leq CM_5h^5$. We obtain for $x\in R_h$
\begin{align*}
L_he_h(x)&=L_hu(x)-L_hu_h(x)\\
&=Lu(x)+ (\rho_3(h)-1)\Delta_c^2 u(x) + h^{-2} \mathcal R_4(x)+\rho_3(h)h^{-4}\mathcal R_5(x)-L_hu_h(x)\\
&=(\rho_3(h)-1)\Delta_c^2 u(x) + h^{-2} \mathcal R_4(x)+\rho_3(h)h^{-4}\mathcal R_5(x).
\end{align*}
For $x\in R^*_h$ we have
\begin{align}\label{eq:looong}
L_{h,2}R_he_h(x)&= L_hR_he_h(x)=L_he_h(x)= (\rho_3(h)-1)\Delta_c^2 u(x) + h^{-2} \mathcal R_4(x)+\rho_3(h)h^{-4}\mathcal R_5(x).
\end{align}
As in the case of $\Delta_c^2$ we have for any $y \in B_h \setminus\partial D$  $$u(y)=u_h(y)+\mathcal R_2(y)$$ where $\lvert \mathcal R_2(y)\rvert\leq C M_2h^2$. Therefore, for $x\in B_h^\ast$,
\begin{align}
L_{h,2} R_he_h(x)&= h^2 L_h R_he_h(x)\nonumber\\
&=h^2 \left\{L_he_h(x)- h^{-4} \sum_{y\in S(x)} \left(h^2 C(y) e_h(y) + \rho_3(h)C^{'}(y) e_h(y)\right)\right\}\nonumber\\
&\stackrel{\eqref{eq:looong}}{=} h^2(\rho_3(h)-1)\Delta_c^2 u(x) + \mathcal R_4(x)
+\rho_3(h)h^{-2}\mathcal R_5(x) \nonumber\\
&\qquad \qquad +C\mathcal R^{'}_2(x) + C h^{-2} \rho_3(h)\mathcal R^{''}_2(x)\label{eq:long}
\end{align}
where $S(x)$ is defined similarly as in $\Delta_c^2$ case, $C(y),\, C^{'}(y)$ are constants depending on $y$ and $\lvert \mathcal R^{'}_2(x)\rvert\leq CM_2h^2,\,\lvert \mathcal R^{''}_2(x)\rvert\leq CM_2h^2$. We have
\begin{align*}
\|L_{h,2} R_he_h\|_{h,grid}^2 &= h^d\sum_{x\in R_h} (L_{h,2} R_he_h(x))^2\\
&= h^d \left[ \sum_{x\in R_h^\ast} (L_{h,2} R_he_h(x))^2+ \sum_{x\in B_h^\ast} (L_{h,2} R_he_h(x))^2 \right]
\end{align*}
which, using the bounds~\eqref{eq:looong}-\eqref{eq:long}, turns into 
\begin{align*}
\|L_{h,2} R_he_h\|_{h,grid}^2
&\le C h^d \sum_{x\in R_h^\ast} \left( (\rho_3(h)-1)^2M_4^2 + M_4^2  h^4 + (\rho_3(h))^2M_5^2h^2 \right)\\
& + Ch^d\sum_{x\in B_h^\ast} \left(h^4(\rho_3(h)-1)^2M_4^2 +M_4^2 h^8 + (\rho_3(h))^2 M_5^2h^6+ M_2^2 + M_2^2 (\rho_3(h))^2h^4 \right)\\
&\le C[(\rho_3(h)-1)^2M_4^2 + M_4^2 h^4 + (\rho_3(h))^2M_5^2h^2 + h^5(\rho_3(h)-1)^2M_4^2 \\
&+M_4^2 h^9 + (\rho_3(h))^2 M_5^2h^7+ M_2^2 h+ M_2^2 (\rho_3(h))^2h^5]
\end{align*}
where in the last inequality we have used that the number of points in $B_h^\ast$ is $O(h^{-(d-1)})$. 
Finally to complete our proof we obtain using Lemma~\ref{lem:3.1} and Lemma~\ref{thm:mixed_4.2}
\begin{align*}\label{eq:mixed_errorbound}
\|R_he_h\|_{h,\,grid}^2 &\le C[(\rho_3(h)-1)^2 M_4^2 + M_4^2 h^4 + (\rho_3(h))^2M_5^2h^2 + h^5(\rho_3(h)-1)^2 M_4^2\\
 &+M_4^2 h^9  + (\rho_3(h))^2 M_5^2h^7+ M_2^2 h+ M_2^2 (\rho_3(h))^2h^5]\\
 & \le C\left[M_5^2h^2 + M_4^2(\rho_3(h)-1)^2  + M_4^2 h^4 + M_2^2 h \right].\qedhere
\end{align*}  
\end{proof}

\subsubsection{Laplacian case: proof of Theorem~\ref{approx_result}~\ref{thm:gff_one}}
In this subsection we consider $L=-\Delta_c$. 
The continuum problem ~\eqref{eqa:continuum} is defined with one boundary condition, whereas in the discrete Dirichlet problem involving $L_h$ two boundary conditions are needed. The contribution of $\Delta_h^2$ is negligible in the limit but for non-zero $h$ it is not. It is the effect of $\rho_1(h)$ which makes $\Delta_h^2$ vanish in the limit. However, if we simply apply the same proof of Theorem~\ref{approx_result}~\ref{thm:mm_one}-\ref{thm:mixed_one}, then if $\rho_1(h)$ does not decay faster than $h$, the method fails to estimate the error. This is due to the fact that one would treat the boundary layer effect and the discretization effect simultaneously. To take care of the different scales at which these effects are seen, we use a suitable cutoff function instead of truncating the discrete operator $L_h$ near the boundary. Using the cutoff we define a function $g$ which is equal to $u$ near the boundary of $D$ and has nice bounds on its derivatives. With the help of $g$ we first take care of the boundary effect. Then we take the discretization parameter $h$ to go to zero and estimate the error.

 Let us first define the cutoff function. Recall that $\delta := \max\{h, \sqrt{\rho_1(h)}\}$. We define
\begin{align*}
D^{\ell\delta} := \{x\in\R^d: \mathrm{dist}(x, \partial D) < \ell\delta\},\quad \ell=1,2,\ldots
\end{align*}
where $\mathrm{dist}(x, \partial D) = \inf \{\|x-y\|: y\in\partial D \} $. Then we have the following Proposition which follows from Theorem 1.4.1 and equation (1.4.2) of~\cite{hormander2015}.
\begin{lemma}
One can find $\phi \in C_c^\infty\left(\overline{D^{7\delta}}\right)$ with $0\le \phi \le 1$ so that $\phi =1$ on $\overline{D^{5\delta}}$ and 
\begin{equation}\label{eq:cutoff_derivative}
\sup_{x\in\R^d} |D^\alpha\phi(x)| \le C_\alpha \delta^{-|\alpha|},
\end{equation}
where $C_\alpha$ depends on $\alpha$ and $d$.
\end{lemma}
We now define a function $g:\overline D \to R$ so that $g= \tilde{\phi} u$ where $\tilde{\phi} $ is the restriction of $\phi$ to $\overline D$. We will use the following bounds of $g$ and its derivatives.
\begin{lemma}\label{lem:bounds_g}
We have
\begin{enumerate}
\item $\displaystyle  \sup_{x\in D}|g(x)| \le CM_1 \delta,$
\item $\displaystyle \sum_{|\alpha| \le 1}\sup_{x\in D}|D^\alpha g(x)| \le CM_1, $
\item $\displaystyle\sum_{|\alpha| \le 2}\sup_{x\in D}|D^\alpha g(x)| \le C(M_1 \delta^{-1} + M_2) .$
\end{enumerate}
Here we recall that $M_k= \sum_{|\alpha| \le k}\sup_{x\in D}|D^\alpha u(x)|$.
\end{lemma}
\begin{proof}
We first observe that $g=0$ on $D\setminus \overline {D^{7\delta}}$. For any $x$ in $D\cap \overline{D^{7\delta}}$ we use Taylor series expansion and the fact that $u=0$ on $\partial D$ to obtain $|u(x)| \le CM_1 \delta$. The bounds now follows from the definition of $g$ and~\eqref{eq:cutoff_derivative}.
\end{proof}
We are now ready to prove Theorem~\ref{approx_result}~\ref{thm:gff_one}.
\begin{proof}[Proof of Theorem~\ref{approx_result}~\ref{thm:gff_one}]
For our convenience we denote by $\|\cdot\|_{\ell^2(A)}$ the $\|\cdot\|_{h,\,grid}$ norm of the projection of any grid-function onto the finite subset $A$ of $h\Z^d$. More precisely, for any finite subset $A$ of $h\Z^d$ and function $v:h\Z^d\to \R$ we define 
\begin{align}
\|v\|_{\ell^2(A)}^2:= h^d\sum_{x\in A} v(x)^2.
\end{align}	
We extend $u $ and $g$ on $\R^d$ by defining their values to be zero outside $\overline D$. Also let us extend $u_h$ by defining it to be zero on $h\Z^d\setminus D_h$. Note that $B_h \subset \overline D \cap \overline{D^{5\delta}}$. Thus by definition we have $e_h =u=g$ on $B_h$. Therefore from Lemma~\ref{lem:3.1} we have
\begin{align}\label{eq:revised_0}
\|R_he_h\|_{h,\,grid}^2 &\le 2 \|e_h-g\|_{\ell^2(R_h)}^2 + 2\|g\|_{\ell^2(R_h)}^2 \nonumber\\
&\le C \|\nabla_h (e_h-g)\|_{\ell^2(R_h\cup \partial R_h)}^2 + 2\|g\|_{\ell^2(R_h)}^2
\end{align}
where $$\nabla_h v(x):= (\partial_j v(x))_{j=1}^d,$$ $$\|\nabla_h v\|_{\ell^2(A)}^2:= \sum_{j=1}^d \|\partial_jv\|_{\ell^2(A)}^2,$$
and $\partial R_h :=\{x\in h\Z^d\setminus R_h: \mathrm{dist}_{h\Z^d}(x,\,R_h)=1\}$ with $\mathrm{dist}_{h\Z^d}$ being the graph distance in the lattice $h\Z^d$.
We have for $x\in R_h$
\begin{align*}
L_h(e_h-g)(x)= L_hu(x)-f(x) -L_hg(x).
\end{align*}
Thus 
\begin{equation}\label{eq:revised_1}
\la L_h(e_h-g), e_h-g\ra_{h,\,grid}= \la L_hu-f, e_h-g\ra_{h,\,grid} + \la -L_hg, e_h-g\ra_{h,\,grid}.
\end{equation}
Using integration by parts we obtain
\begin{equation}\label{eq:revised_2}
\la L_h(e_h-g), e_h-g\ra_{h,\,grid} = \|\nabla_h(e_h-g)\|_{\ell^2(R_h\cup \partial R_h)}^2 + \rho_1(h) \|\Delta_h(e_h-g)\|_{\ell^2(R_h\cup \partial R_h)}^2.
\end{equation}
For the first term in equation~\eqref{eq:revised_1} we have, using Lemma~\ref{lem:3.1},
\begin{align}\label{eq:revised_3}
|\la L_hu-f, e_h-g\ra_{h,\,grid}|& \le \|L_hu-f\|_{\ell^2(R_h)} \|e_h-g\|_{\ell^2(R_h)} \nonumber\\
&\le C \|L_hu-f\|_{\ell^2(R_h)} \|\nabla_h(e_h-g)\|_{\ell^2(R_h\cup \partial R_h)} \nonumber\\
&\le C \|L_hu-f\|^2_{\ell^2(R_h)} + \frac14 \|\nabla_h(e_h-g)\|_{\ell^2(R_h\cup \partial R_h)}^2.
\end{align}
For the second term of equation~\eqref{eq:revised_1} we obtain using integration by parts
\begin{align}\label{eq:revised_4}
|\la -L_hg, e_h-g\ra_{h,\,grid}| &\le |\la -\Delta_hg, e_h-g\ra_{h,\,grid}| + \rho_1(h)|\la \Delta_h^2g, e_h-g\ra_{h,\,grid}| \nonumber\\
& \le |\la \nabla_hg, \nabla_h(e_h-g)\ra_{h,\,grid}| + \rho_1(h)|\la \Delta_hg, \Delta_h(e_h-g)\ra_{h,\,grid}|\nonumber\\
& \le \|\nabla_hg\|^2_{\ell^2(R_h\cup \partial R_h)} + \frac14 \|\nabla_h(e_h-g)\|^2_{\ell^2(R_h\cup \partial R_h)} + \rho_1(h)\|\Delta_hg\|^2_{\ell^2(R_h\cup \partial R_h)}\nonumber\\& + \rho_1(h)\|\Delta_h(e_h-g)\|^2_{\ell^2(R_h\cup \partial R_h)}.
\end{align}
Combining~\eqref{eq:revised_1},~\eqref{eq:revised_2},~\eqref{eq:revised_3} and~\eqref{eq:revised_4} we get
\begin{align*}
\|\nabla_h(e_h-g)\|^2_{\ell^2(R_h\cup \partial R_h)} \le C \|L_hu-f\|^2_{\ell^2(R_h)} + C \|\nabla_hg\|^2_{\ell^2(R_h\cup \partial R_h)} + C\rho_1(h)\|\Delta_hg\|^2_{\ell^2(R_h\cup \partial R_h)}.
\end{align*}
Substituting this in~\eqref{eq:revised_0} we obtain
\begin{align}\label{eq:revised_main}
\|R_he_h\|_{h,\,grid}^2 &\le C \|L_hu-f\|^2_{\ell^2(R_h)} + C \|\nabla_hg\|^2_{\ell^2(R_h\cup \partial R_h)}\nonumber\\
& + C\rho_1(h)\|\Delta_hg\|^2_{\ell^2(R_h\cup \partial R_h)} + 2\|g\|_{\ell^2(R_h)}^2.
\end{align}
We now bound each of the term in the right hand side of the inequality~\eqref{eq:revised_main}. Using Taylor series expansion we have for all $x\in R_h$
$$L_hu(x)=Lu(x)+ h^{-2}\mathcal R_4(x) + h^{-4}\rho_1(h) \mathcal R_4^{'}(x)$$
where $\lvert \mathcal R_4(x)\rvert\leq CM_4h^4$ and $\lvert \mathcal R_4^{'}(x)\rvert\leq CM_4h^4$. Now
\begin{align*}
\|L_hu-f\|^2_{\ell^2(R_h)} & \le h^d \sum_{x\in R_h} (M_4^2 h^4 + M_4^2 \rho_1(h)^2 ) \le C M_4^2 \delta^4.
\end{align*}
For the second term of~\eqref{eq:revised_main} we have the bound
\begin{align*}
\|\nabla_hg\|^2_{\ell^2(R_h\cup \partial R_h)} &= h^d\sum_{x\in (R_h \cup\partial R_h)\cap \overline{D^{8\delta}}} h^{-2}\sum_{i=1}^d (g(x+he_i)-g(x))^2\\
&\le C h^d \sum_{x\in (R_h \cup\partial R_h)\cap \overline{D^{8\delta}}} M_1^2 \le CM_1^2 \delta
\end{align*}
where in the first inequality we used Taylor expansion and Lemma~\ref{lem:bounds_g} and in the last inequality we used the fact that number of points in $(R_h \cup\partial R_h)\cap \overline{D^{8\delta}}$ is $O(\delta h^{-d})$. Similarly, for the third term using Taylor expansion, Lemma~\ref{lem:bounds_g} and the fact that number of points in $(R_h \cup\partial R_h)\cap \overline{D^{8\delta}}$ is $O(\delta h^{-d})$ we have
\begin{align*}
\rho_1(h)\|\Delta_hg\|^2_{\ell^2(R_h\cup \partial R_h)}&= \rho_1(h)h^d\sum_{x\in (R_h \cup\partial R_h)\cap \overline{D^{8\delta}}} (\Delta_hg(x))^2\\
&\le C \rho_1(h)h^d\delta h^{-d}(M_1 \delta^{-1} + M_2)^2\\
&\le C\left( M_1^2 \sqrt{\rho_1(h)} + M_2^2 \rho_1(h)\delta \right).
\end{align*}
Finally we obtain
\begin{align*}
\|g\|_{\ell^2(R_h)}^2 & = h^d\sum_{x\in R_h\cap \overline{D^{7\delta}}} g(x)^2 \\
& \le C h^d\sum_{x\in R_h\cap \overline{D^{7\delta}}} M_1^2 \delta^2\\
& \le C M_1^2 \delta^3.
\end{align*}
Here in the first inequality we used Lemma~\ref{lem:bounds_g} and in the last inequality we used the fact that number of points in $R_h \cap \overline{D^{7\delta}}$ is $O(\delta h^{-d})$. Combining all these bounds we obtain from~\eqref{eq:revised_main}
\begin{align*}
\|R_he_h\|_{h,\,grid}^2 &\le C\left(M_4^2 \delta^4 + M_1^2 \delta +  M_1^2 \sqrt{\rho_1(h)} + M_2^2 \rho_1(h)\delta + M_1^2 \delta^3\right) \\
&\le  C\left(M_4^2 \delta^4 + M_2^2 \rho_1(h)\delta  + M_1^2 \delta \right).\qedhere
\end{align*}

\end{proof}

\appendix 
\section{Covariance bound for the membrane model in $d=1$}\label{appendix:A}
In this section we consider $d=1$ and the membrane model $(\vr_x)_{x\in V_N}$ on $V_N = \{1,\ldots, N-1\}$ with zero boundary conditions outside $V_N$. We want to show the following bound:
\begin{lemma}\label{lem:one_mm}
There exists a constant $C>0$ such that
\begin{align*}
\E_{V_N} [(\vr_x-\vr_{x+1})^2] \le CN, \quad x\in \Z.
\end{align*}
\end{lemma} 
\begin{proof}
Let $\{X_i\}_{i\in\N}$ be a sequence of i.i.d. standard Gaussian random variables. We define $\{Y_i\}_{i\in\Z^+}$ to be the associated random walk starting at $0$, that is,
\[ Y_0=0,\quad Y_n=\sum_{i=1}^n X_i,\,\,n\in\N,\]
and $\{Z_i\}_{i\in\Z^+}$ to be the integrated random walk starting at $0$, that is, $Z_0=0$ and for $n\in\N$
\[Z_n=\sum_{i=1}^n Y_i.\]
Then one can show that $\prob_{V_N}$ is the law of the vector $(Z_1,\ldots,Z_{N-1})$ conditionally on $Z_N=Z_{N+1}=0$ \cite[Proposition 2.2]{CaravennaDeuschel_pin}. 
So we have that
$$\E_{V_N}\left[ (\vr_{i+1}-\vr_i)^2\right]= \E\left[ (Z_{i+1}-Z_i)^2| Z_N=Z_{N+1}=0\right]=\E\left[ Y_{i+1}^2| Z_N=Z_{N+1}=0\right].$$
Hence it is enough to find a bound for  $\E[Y_i^2|Z_N=Z_{N+1}=0]$ for $i=1,\ldots,N-1$. 
The covariance matrix $\Sigma$ for $(Y_1,\ldots, Y_{N-1}, Z_{N}, Z_{N+1})$ can be partitioned as
$$\Sigma= \begin{bmatrix}
 A & B \\
 B & D 
\end{bmatrix}$$
where $A$ is a $(N-1) \times (N-1)$ matrix with entries
$$A(i,j)= \cov( Y_i, Y_j)= \min\{i,\,j\}.$$
$B(i,j)$ and $C(i,j)$ are $(N-1)\times 2$ and $2\times (N-1)$ matrices respectively, with $C=B^T$ and
$$B(i,j)= \cov(Y_i,Z_{j+N-1})= \sum_{l=1}^{j+N-1} \min\{i,\,l\}.$$
Finally, $D$ is a $2\times 2$ matrix with 
$$D(i,j)= \cov(Z_{i+N-1},Z_{j+N-1}).$$
It easily follows that
\begin{equation}\label{eq:Dmatrix}
D=\frac{1}{6}\begin{bmatrix}
N(N+1) (2N+1) & N(N+1)(2N+4)\\
N(N+1)(2N+4) & (N+1)(N+2)(2N+3)
\end{bmatrix}.
\end{equation}

It is well known that $(Y_1,\ldots,Y_{N-1}|Z_N=Z_{N+1}=0)$ is a Gaussian vector with mean zero and covariance matrix given by $A-BD^{-1}C$. The inverse of $D$ is as follows.
Observe$$\gamma_N:= \det(D)= \frac{1}{36} N(N+1)^2(8N^2+3N+6)$$
and 
$$D^{-1}= \frac1{\gamma_N}\begin{bmatrix}
D(2,2)\quad -D(1,2)\\
-D(2,1)\quad D(1, 1)
\end{bmatrix}$$

%
Now the diagonal element of $BD^{-1}C$ can be determined:
\begin{align*}
(BD^{-1}C)(i,i)&=\frac1{\gamma_N}\left[\left(\sum_{l=1}^{N} \min\{i,\,l\}\right)^2D(2,2)-\left(\sum_{l=1}^{N} \min\{i,\,l\}\right)\left(\sum_{l=1}^{N+1} \min\{i,\,l\}\right) D(1,2)\right.\\
&\left.-\left(\sum_{l=1}^{N} \min\{i,\,l\}\right)\left(\sum_{l=1}^{N+1} \min\{i,\,l\}\right) D(1,2) + \left(\sum_{l=1}^{N+1} \min\{i,\,l\}\right)^2D(1,1)\right].
\end{align*}
Plugging in the entries $D(i,j)$ from \eqref{eq:Dmatrix} and simplifying we get
$$(BD^{-1}C)(i,i)=\frac{i^2(N+1)}{24\gamma_N}\left[6N^2-12Ni+6i^2+4N\right]>0$$ 
This shows that for $i=1,2,\ldots, N-1$,
 $$\E[Y_i^2|Z_N=Z_{N+1}=0]=A(i,i)-(BD^{-1}C)(i,i)<i.$$ 
Similar bound can be obtained for $\E[Y_N^2|Z_N=Z_{N+1}=0]$ and this completes the proof. 
\end{proof}

\section{Details on the space $\mathcal H^{-s}_{-\Delta+\Delta^2}(D)$}\label{appendix:B}
In this section we briefly describe some of the details regarding the space $\mathcal H^{-s}_{-\Delta+\Delta^2}(D)$ and also about the spectral theory of $-\Delta_c+\Delta^2_c$. This is an elliptic operator, and the spectral theory is similar to that of either $-\Delta_c$ or $\Delta_c^2$.
First recall the standard Sobolev inner products on $H^1_0(D)$ and $H^2_0(D)$. They are
$$\left\langle u, v\right\rangle_{1}= \int_D \nabla u\cdot \nabla v \De x,\quad u,\,v\in H^1_0(D)$$ and 
$$\left\langle u, v\right\rangle _{2}= \int_D \Delta u \Delta v \De x,\quad u,\,v\in H^2_0(D)$$
and they induce norms on $H^1_0(D)$ and $H^2_0(D)$ respectively which are equivalent to the standard Sobolev norms~\cite[Corollary~2.29]{GGS}. We now consider the following inner product on $H^2_0(D)$: 
$$\left\langle u, v\right\rangle _{mixed}:= \int_D \nabla u\cdot \nabla v \De x + \int_D \Delta u \Delta v \De x,\,\,u,v\in H^2_0(D).$$ 
Clearly the norm induced by this inner product is equivalent to the norm $\|\cdot\|_{H^2_0}$ (by integration by parts). We consider $H^{-2}(D)$ to be the dual of $(H^2_0(D),\,\|\cdot\|_{mixed})$.

We now give some results whose proofs are similar to Theorem 3.2 and 3.3 of \cite{mm_scaling}.
\begin{enumerate}
\item There exists a bounded linear isometry
	$$T_0:H^{-2}(D)\rightarrow (H^2_0(D),\,\|\cdot\|_{mixed})$$ 
	such that, for all $f\in H^{-2}(D)$ and for all $v\in H^2_0(D)$, $$(f\,,\,v)= \la v,\,T_0 f\ra_{mixed}.$$ Moreover, the restriction $T$ on $L^2(D)$ of the operator $i\circ T_0 :H^{-2}(D)\rightarrow L^2(D) $ 
	is a compact and self-adjoint operator, where $i: (H^2_0(D),\,\|\cdot\|_{mixed}) \hookrightarrow L^2(D)$ is the inclusion map.
	
\item There exist $v_1,\, v_2,\, \ldots $ in $(H_0^2(D),\|\cdot\|_{mixed})$ and numbers $0<\mu_1\le\mu_2\le \cdots \to\infty$ such that
	\begin{itemize}
		\item $\{v_j\}_{j\in \N}$ is an orthonormal basis for $L^2(D)$,
		\item $Tv_j=\mu^{-1}_jv_j$,
		\item $\la v_j,v \ra_{mixed}=\mu_j\la v_j, v \ra_{L^2}$  for all $v\in H^2_0(D)$,
		\item $\{\mu^{-1/2}_jv_j\}$ is an orthonormal basis for $(H_0^2(D),\|\cdot\|_{mixed})$.
	\end{itemize}
\end{enumerate}

For each $j\in \N$ one has $v_j\in C^\infty(D).$ Moreover $v_j$ is an eigenfunction of $-\Delta_c + \Delta_c^2$ with eigenvalue $\mu_j$. Indeed, we have for all $v\in H_0^2(D)$
	\begin{align*} 
	\la (-\Delta_c + \Delta_c^2) v_j,\,v \ra_{L^2} = \la (-\Delta_c) v_j,\,v \ra_{L^2} + \la (\Delta_c^2) v_j,\,v \ra_{L^2}\stackrel{GI}{=}\la v_j,\,v \ra_{mixed} =\mu_j\la v_j,\,v \ra_{L^2}
	\end{align*}
	where ``GI'' stands for Green's first identity
	\[
	\int_D u\Delta v \De V=-\int_D\nabla u\cdot\nabla v \De V+\int_{\partial D}u\nabla v\cdot\mathbf n \De S.
	\]
	Thus $v_j$ is an eigenfunction of $-\Delta_c + \Delta_c^2$ with eigenvalue $\mu_j$ in the weak sense. The smoothness of 
	$v_j$ follows from the fact that $-\Delta_c + \Delta_c^2$ is an elliptic operator with smooth coefficients 
	and the elliptic regularity theorem \cite[Theorem~9.26]{FollandReal}. Hence $v_j$ is an 
	eigenfunction of $-\Delta_c + \Delta_c^2$ with eigenvalue $\mu_j$.
As a consequence of the above, one easily has that
	%
	\begin{equation}\label{rem:mixed_series} \|f\|_{mixed}^2 = \sum_{j\ge 1} \mu_j \la f, v_j\ra^2_{L^2}\end{equation}
	for any $f\in H^2_0(D)$.

For any $v\in C_c^\infty(D)$ and for any $s>0$ we define $$\|v\|_{s, \,-\Delta + \Delta^2}^2:=\sum_{j\in \N}\mu_j^{s/2}\la v,v_j\ra_{L^2}^2.$$ We define $\mathcal H_{-\Delta+\Delta^2, 0}^{s}(D)$ to be the Hilbert space completion of $C_c^\infty(D)$ 
with respect to the norm $\|\cdot\|_{s, \,-\Delta + \Delta^2}$.
Then $\left(\mathcal H_{-\Delta+\Delta^2, 0}^{s}(D) \,,\,\|\cdot\|_{s, \,-\Delta + \Delta^2}\right) $ is a Hilbert space for all $s>0$.
Moreover, we also notice the following.
	\begin{itemize}
		\item Note that for $s=2$ we have $\mathcal H_{-\Delta+\Delta^2, 0}^{2}(D)= (H_0^2(D),\,\|\cdot\|_{mixed})$ by~\eqref{rem:mixed_series}.
		\item $i:\mathcal H_{-\Delta+\Delta^2, 0}^{s}(D)\hookrightarrow L^2(D)$ is a continuous embedding.
	\end{itemize}
%
For $s>0$ we define $\mathcal H^{-s}_{-\Delta+\Delta^2}(D)= (\mathcal H_{-\Delta+\Delta^2, 0}^{s}(D))^*$, the dual space of $\mathcal H_{-\Delta+\Delta^2, 0}^{s}(D)$. Then we have $$\mathcal H_{-\Delta+\Delta^2, 0}^s(D) \subseteq L^2(D) \subseteq  \mathcal H^{-s}_{-\Delta+\Delta^2}(D).$$
One can show using the Riesz representation theorem that for $s>0$, and $v\in L^2(D)$ the norm of $\mathcal H^{-s}_{-\Delta+\Delta^2}(D)$ is given by
\[
\|v\|_{-s, \,-\Delta + \Delta^2}^2:=\sum_{j\in \N}\mu_j^{-s/2}\la v,\,v_j\ra_{L^2}^2. 
\]

Before we show the definition of the continuum mixed model, we need an analog of Weyl's law for the eigenvalues of the operator $-\Delta_c + \Delta_c^2$. 
\begin{proposition}[{\citet[Theorem~5.1]{Beals:1967}, \citet{Pleijel:1950}}]\label{prop:mixed_Weyl}
	There exists an explicit constant $c$ such that, as $j\uparrow+\infty$,
	\[
	\mu_j\sim c^{-d/4}j^{4/d}.
	\]
\end{proposition}
\begin{proof}
We want to apply Theorem 5.1 of \cite{Beals:1967} for $A:= -\Delta_c + \Delta_c^2$. First note that $A$ is an elliptic operator of order $m=4$ defined on $D$ having smooth coefficients. Let us consider $A_1:= (-\Delta_c+\Delta_c^2)|_{H^4(D)\cap H^2_0(D)}$.
Clearly, $A_1: H^4(D)\cap H^2_0(D) \rightarrow L^2(D)$ and also $C_c^\infty(D) \subset D(A_1) \subset H^4(D)$, where $D(A_1)$ is the domain of $A_1$. By elliptic regularity we have $D(A_1^p) \subset H^{4p}, \, p = 1,\,2,\,\ldots$ We first show that $A_1$ is self-adjoint. Note that as $C_c^\infty(D) \subset D(A_1) $ and $C_c^\infty(D)$ is dense in $L^2(D)$, $A_1$ is densely defined.
Again, by Green's identity we have for all $u,v \in H^4(D)\cap H^2_0(D)$
\begin{align*}
\la (-\Delta_c + \Delta_c^2)u,\, v\ra_{L^2} = \la \nabla u,\, \nabla v\ra_{L^2} + \la \Delta_c u,\, \Delta_c v\ra_{L^2} = \la  u,\, (-\Delta_c + \Delta_c^2)v\ra_{L^2}.
\end{align*}
Thus $A_1$ is symmetric. Also by Corollary 2.21 of \citet{GGS} we observe that image of $A_1$ is $ L^2(D)$. The self-adjointness of $A_1$ now follows from Theorem 13.11 of \cite{rudin_functional}. Also we conclude from Theorem 13.9 of \cite{rudin_functional} that $A_1$ is closed. 
Now applying Theorem 5.1 of \cite{Beals:1967} we get the asymptotic.
\end{proof}
The result we will prove now shows the well-posedness of the series expansion for $\Psi^{-\Delta+\Delta^2}$.
\begin{proposition}\label{prop:mixed_series_rep_h}
	Let $(\xi_j)_{j\in \N}$ be a collection of i.i.d. standard Gaussian random variables. Set
	\[
	\Psi^{-\Delta+\Delta^2}:=\sum_{j\in \N}\mu_j^{-1/2}\xi_j v_j.
	\]
	Then $\Psi^{-\Delta+\Delta^2}\in \mathcal H^{-s}_{-\Delta+\Delta^2}(D)$ a.s. for all $s>({d-4})/2$.
\end{proposition}
\begin{proof}
	Fix $s>({d-4})/2$. Clearly $v_j\in L^2(D)\subseteq \mathcal H^{-s}_{-\Delta+\Delta^2}(D)$. We need to show that $\|\psi\|_{-s, \,-\Delta + \Delta^2}<+\infty$ almost surely. Now this boils down to showing the finiteness of the random series
	$$\|\psi\|_{-s, \,-\Delta + \Delta^2}^2=\sum_{j\ge 1} \mu_j^{-s/2} \left(\sum_{k\ge 1} \mu_k^{-1/2} u_k \xi_k\, , v_j\right)^2=\sum_{j\ge 1} \mu_j^{-\frac{s}{2}-1}\xi_j^2 $$
	where the last equality is true since $(v_j)_{j\ge 1}$ form an orthonormal basis of $L^2(D)$. Observe that the assumptions of Kolmogorov's two-series theorem are satisfied: indeed using Proposition~\ref{prop:mixed_Weyl} one has 
	\[\sum_{j\ge 1}\E\left(\mu_j^{-\frac{s}{2}-1}\xi_j^2\right)\asymp \sum_{j\ge 1} j^{-\frac{4}{d}\left(\frac{s}{2}+1\right)}<+\infty\]
	for $s>(d-4)/2$ and 
	\[\sum_{j\ge 1}\var\left(\mu_j^{-\frac{s}{2}-1}\xi_j^2\right)\asymp \sum_{j\ge 1} j^{-\frac{4}{d}(s+2)}<+\infty\] 
	for $s>(d-8)/4$. The result then follows.
\end{proof}

\section{Random walk representation of the $(\nabla+\Delta)$-model in $d=1$ and estimates}\label{appendix:C}
In this Appendix we recall some of the notations about the $d=1$ case which were used in the heuristic explanations of the Introduction. We take advantage of the representation of the mixed model given in \citet[Subsection 3.3.1]{borecki2010} in our setting. To do that we set $\beta_N:= 16
\kappa_N$.

Let \begin{equation}\label{eq:gamma}
    \gamma= \left(\frac{1+\beta_N-\sqrt{1+2\beta_N}}{1+\beta_N+\sqrt{1+2\beta_N}}\right)^{1/2}
\end{equation}
and let $(\eps_i)_{i\in \Z^{+}}$ be i.i.d. $\mathcal N(0,\sigma^2)$ with 
\begin{equation}\label{eq:sigma}
\sigma^2= 4/(1+\beta_N+\sqrt{1+2\beta_N}).
\end{equation}
Define
$$Y_n= \gamma^{n-1}\eps_1+\ldots+ \gamma^0 \eps_n= \sum_{i=1}^{n}\gamma^{n-i}\eps_i.$$
Let the integrated walk be denoted by
$$W_n=\sum_{i=1}^n Y_i= r_{n-1}\eps_1+\ldots+r_0 \eps_n= \sum_{i=1}^n r_{n-i}\eps_i$$
where $r_{n-i}= \sum_{i=0}^{n-i} \gamma^i$.

We consider the case when $\kappa_N\to \infty$ and note that then $\gamma=\gamma_N\to 1$ and $\sigma_N^2=\sigma^2\to 0$. 
The following representation will give an idea on how the phase transition occurs in the mixed model:
$$W_n= \frac{1}{1-\gamma}(\eps_1+\cdots+\eps_n)-\frac{1}{1-\gamma}(\gamma^n\eps_1+\gamma^{n-1}\eps_2+\cdots+\gamma \eps_n).$$
We recall the following Proposition from \citet[Proposition 1.10]{borecki2010}.
\begin{proposition} 
Let $\prob_{N}(\cdot)$ be the mixed model with $0$ boundary conditions. Then
$$\prob_{N}(\cdot)= \prob\left( (W_1,\ldots, W_{N-1})\in \cdot| W_N=W_{N+1}=0\right)$$
\end{proposition}
Let $(\widetilde \eps_i)_{i\in \Z^{+}}$ be i.i.d.\ $\mathcal N\left(0, \frac{\sigma^2}{(1-\gamma)^2}\right)$. Then $W_n$ can be written as
$$W_n= S_n-U_n$$
where $S_n=\sum_{k=1}^n \widetilde \eps_k$ and
$U_n= \gamma^n \widetilde \eps_1+\gamma^{n-1} \widetilde \eps_2+\cdots+ \gamma  \widetilde \eps_n $.
The conditional integrated random walk process has a representation, stated in Proposition 3.7 of \cite{borecki2010}. 
Let
$$\prob\left((\widehat W_1, \ldots, \widehat W_{N-1})\in \cdot\right)=\prob\left( (W_1,\ldots, W_{N-1})\in \cdot| W_N=W_{N+1}=0\right).$$
Then
$$\widehat W_k= W_k- W_Nr_1(k)-W_{N+1}r_2(k)$$
where $r_1(k)= s_1(k)/r(k)$ and $r_2(k)=s_2(k)/r(k)$.
The definitions of $r(k)$ and $s_i(k)$ for $i=1,2$ are as follows:
\begin{align*}
&r(k)=(-1+\gamma)(-1+\gamma^{N+1})\left(-N+\gamma(2+N+\gamma^N(-2+(-1+\gamma)N))\right),\\
    &s_1(k)= (-k+\gamma(1-\gamma^k+k))+\gamma^{3+2N+k}(1+\gamma^k(-1+(-1+\gamma)k))\\
    &+\gamma^{N-k}(\gamma^k(-\gamma+\gamma^3)(1-k+N)+\gamma^{2+2k}(2+N-\gamma(1+N))+\gamma(1+N-\gamma(2+N))),
\end{align*}
and
\begin{align*}
    &s_2(k)=\gamma(\gamma^{1+k}+k-\gamma(1+k))+\gamma^{2+2N-k}(-1+\gamma^k(1+k-\gamma k))\\
    &+\gamma^{1+N-k}(\gamma+\gamma^k(-1+\gamma^2)(k-N)-N+\gamma N+\gamma^{1+2k}(-1+(-1+\gamma)N)).
\end{align*}
Let us consider the unconditional process $W_n$. Note that
$$ \var( S_n)= \frac{n\sigma^2}{(1-\gamma)^2} ,\quad\quad \var(U_n)= \frac{\sigma^2\gamma^2 (1-\gamma^{2n})}{(1-\gamma)^2(1-\gamma^2)} $$
and 
$$\cov(S_n, U_n)= \frac{\gamma\sigma^2(1-\gamma^n)}{(1-\gamma)^2(1-\gamma)}.$$
So from here we have
\begin{equation}\label{eq:varW}
\var(W_n)= \frac{n\sigma^2}{(1-\gamma)^2} -\frac{\sigma^2\gamma^2(1-\gamma^n)^2}{(1-\gamma)^3(1+\gamma)}-\frac{2\sigma^2\gamma(1-\gamma^N)}{(1-\gamma)^3(1+\gamma)}.
\end{equation}

From the above expressions one can show that $\var(W_{N-1})\sim N$ when $\kappa=\kappa_N\ll N^2$. We now derive the variance estimate when $\kappa\gg N^2$. For ease of writing, denote 
$$\zeta=\frac{1}{\beta_N}+\sqrt{ \frac{1}{\beta_N}}\sqrt{\frac{1}{\beta_N}+2}\to 0.$$
Furthermore $\gamma=1/(1+\zeta)$ and $\sigma^2= 2/\beta_N(1+\zeta)$. Rewriting~\eqref{eq:varW} in terms of $\zeta$ we have
\begin{align}
\var(W_{N-1})&= \frac{2(N-1)(1+\zeta)^2}{\zeta^2\beta_N(1+\zeta)}-\frac{2(1+\zeta)(1-(1+\zeta)^{-(N-1)})^2}{\beta_N\zeta^3(2+\zeta)}- 4\frac{(1+\zeta)^2(1-(1+\zeta)^{-(N-1)})}{\beta_N\zeta^3(2+\zeta)}\nonumber\\
&= \frac{2(1+\zeta)}{\beta_N(2+\zeta)\zeta^3}\left[ (N-1)(2+\zeta)\zeta-(1-(1+\zeta)^{-(N-1)})^2-2(1+\zeta)(1-(1+\zeta)^{-(N-1)})\right]\label{eq:above}.
\end{align}
Using a Taylor series expansion of the fourth order for the second and third summands in~\eqref{eq:above} (since coefficients up to $\zeta^2$ get cancelled) we obtain that
$$\var(W_{N-1})\approx \frac{(1+\zeta) N(N-1)^2}{\beta_N(2+\zeta)}\approx \frac{N^3}{\beta_N} \approx \frac{N^3}{\kappa_N}.$$

\bibliographystyle{abbrvnat}
\bibliography{biblio}

\end{document}